\RequirePackage{cmap} 
\documentclass[12pt,reqno]{amsart}
\usepackage[utf8]{inputenc}
\usepackage[T2A]{fontenc}
\usepackage{amsmath,amssymb,amsthm,blkarray,graphicx,mathtools,footmisc,bbm}
\usepackage[a4paper,hmargin=2.5cm,vmargin=2.5cm]{geometry}
\usepackage[english]{babel}
\usepackage[mathscr]{euscript}
\usepackage{tikz}
\usepackage{tikz-cd}
\usetikzlibrary{braids}
\usepackage{leftindex}
\usetikzlibrary{arrows.meta,decorations.markings}
\usepackage{blkarray}
\usepackage{listings}

\definecolor{darkspringgreen}{rgb}{0.09, 0.45, 0.27}

\usepackage[pdftex,bookmarks=false,colorlinks=black,citecolor=darkspringgreen,debug=true,
  naturalnames=true,pdfnewwindow=true]{hyperref}

\newtheorem{thm}{Theorem}[section]

\newtheorem{cor}[thm]{Corollary}

\newtheorem{lem}[thm]{Lemma}
\newtheorem{prop}[thm]{Proposition}
\newtheorem{conj}[thm]{Conjecture}

\theoremstyle{definition}
\newtheorem{defn}[thm]{Definition}
\newtheorem{ex}[thm]{Example}

\theoremstyle{remark}
\newtheorem{rem}[thm]{Remark}

\newcommand{\thmref}[1]{Theorem~\ref{#1}}

\numberwithin{equation}{section}

\newcommand{\nc}{\newcommand}
\nc{\on}{\operatorname} \nc{\wh}{\widehat}
\nc\ol{\overline} \nc\ul{\underline} \nc\wt{\widetilde}
\nc{\Hom}{\operatorname{Hom}} \nc{\reg}{\operatorname{reg}} \nc{\diag}{\operatorname{diag}}
\nc{\red}[1]{{\color{red}#1}}
\nc{\blue}[1]{{\color{blue}#1}}

\newcommand{\CC}{\mathbb{C}}

\newcommand{\GL}[1]{GL_{#1}(\mathbb{F}_q)}

\DeclareMathOperator{\im}{im}
\DeclareMathOperator{\tr}{Tr}
\DeclareMathOperator{\End}{End}

\DeclareMathOperator{\R}{\mathcal{R}}

\tikzset{middlearrow/.style={
        decoration={markings,
            mark= at position 0.55 with {\arrow{#1}},
        },
        postaction={decorate}
    }
}

\newcommand{\bbone}{\text{\usefont{U}{bbold}{m}{n}1}}
\MakeRobust{\bbone}

\title{Radon transform for $GL_n(\mathbb{F}_q)$}
\author{Ivan Motorin}
\address{Department of Mathematics, Massachusetts Institute of Technology, Cambridge, MA 02139-4307, USA}
\email{ivanm597@mit.edu}

\author{Kai Yamashita}
\address{PRIMES-USA}
\email{kaiyamashita08@gmail.com}

\begin{document}

\begin{abstract}
    In this paper, we study the Radon transform associated with unipotent radical subgroups of $GL_n(\mathbb{F}_q)$. We analyze properties of the Radon transform with a specific emphasis on its eigenvalues. 
    We provide a description of its eigenvalues for the cases of $\GL{2}$ and $\GL{3}$.
\end{abstract}
\maketitle

\tableofcontents

\section{Introduction}
Let $GL_n (\mathbb{F}_q)$ denote the general linear group of a vector space $\mathbb{F}_q^n$ of dimension $n$ over the finite field $\mathbb{F}_q$. Equivalently, these are all invertible $n\times n$ matrices with coefficients in $\mathbb{F}_q$.\\ 

We denote by $U_\lambda$ the unipotent radical of the parabolic subgroup corresponding to the ordered partition $\lambda$ of $n$ in $GL_n (\mathbb{F}_q)$ (as in Example \ref{PLU_examples}). When viewed as a matrix subgroup, these are all the upper block triangular matrices with identities on the diagonal, where the block sizes are determined by $\lambda$. Also, by $M_{\lambda}$ we denote the corresponding Levi subgroup and by $U_{\lambda}^{op}$ we denote the matrix subgroup transposed to the $U_{\lambda}$.\\


In general, Radon transform for a finite group $G$ is a certain $G$-equivariant map of $G$-modules associated with two finite $G$-sets $X$ and $Y$ (see \cite{O}). Specifically, if we have a $G$-equivariant incidence relation $\sim$ between $X$ and $Y$, then we may construct a map of $G$-modules
\[
R: \CC[X] \rightarrow \CC[Y],\qquad R(x)=\sum_{y: x\sim y} y.
\]
Similarly, we can define an adjoint map
\[
R': \CC[Y] \rightarrow \CC[X], \qquad R'(y)=\sum_{x:x\sim y}x.
\]
Then the map $R'\circ R$ is a $G$-module endomorphism of $\CC[X]$ to itself.\\

For the context of our paper the Radon transform is a natural map between function spaces of $\CC$-valued functions on the cosets $\GL{n} / U_\lambda$ and $\GL{n} / U_\lambda^{op}$ which intertwines the action of $\GL{n} \times M_\lambda$ (e.g. see \cite{S}) defined by

\begin{equation*}
    f \mapsto R(f) := \left \{ x \mapsto \sum_{\overline{u}\in U_\lambda^{op}} f(x \overline{u}) \right \}
\end{equation*}

and

\begin{equation*}
    f \mapsto R'(f) := \left \{ x \mapsto \sum_{u \in U_\lambda} f(xu) \right \}.
\end{equation*}

One can consider the associated operators $R,R'$ on the subspaces $\CC [ \GL{n} / U_\lambda]$ and $\CC [ \GL{n} / U_\lambda^{op} ]$ of right $U_\lambda$ invariants and $U_\lambda^{op}$ invariants respectively of the group algebra $\CC[\GL{n}]$. 
We will study the operator $\R=R'\circ R$, which is a map $\CC [ \GL{n} / U_\lambda] \to \CC [ \GL{n} / U_\lambda ]$. Specifically, we want to find the eigenvalues of this operator.\\


The solution to the eigenvalue problem is known for a particular case of the standard Borel subgroup of upper triangular matrices in $\GL{n}$. In \cite{L} G. Lusztig computed the action of the $T_{w_0}$ operator on the Hecke algebra explicitly in terms of the canonical basis. It is possible to show that the Radon transform corresponding to $\lambda_n:=1^{(n)}$ acting on the corresponding Iwahori-Hecke algebra is proportional to $T_{w_0}^2$. It follows that all the eigenvalues of $\R$ will be of the form $q^{2i}$ for some nonnegative integers $i$.\\

The main results of our paper are contained in Theorems \ref{eigenvalues_box}, \ref{eigenvalues_gl2}, \ref{eigenvalues_gl3_Yok}, \ref{eigenvalues_GGH_gl3}, Propositions \ref{Ariki_Koike_Yok}, \ref{max_parabolic} and Corollary \ref{1_1_1} which allow us to describe eigenvalues of all possible variants of Radon transform associated with the unipotent radical of different types in $\GL{n}$ for the cases when $n=2$ and $3$, and we also prove miscellaneous results regarding the operator $\R$ in Sections 3 and 7.\\




The paper is structured as follows. In Section 2 we go over some basic facts from representation theory of finite groups as well as a few structural results concerning general linear groups. We also define the Radon transform in the same section. In Section 3 we firstly prove that $\R$ has real eigenvalues the order of which is divisible by $q-1$. Secondly, we prove that the eigenvalue problem can be reduced to the case of $\lambda=(k,n-k)$. In Section 4 we prove that the eigenvalues of the only possible Radon transform for $n=2$ are equal to $1,q$ and $q^2$, and we find the multiplicities of those eigenvalues in the space of right $U_{(1,1)}$-invariant elements of the group algebra. In Section 5, we give a brief overview of Yokonuma-Hecke algebras and Gelfand-Graev Hecke algebras. In Section 6, we describe the eigenvalues for the Radon transforms of type $\lambda=(1,2),(2,1)$ or $(1,1,1)$ (the case of $n=3$) by restricting $\R$ to the Yokonuma-Hecke algebra or a certain class of Gelfand-Graev Hecke algebras as well as by utilizing our result from Section 4. Additionally, we provide a description of the eigenvalues of the Radon transform in the Yokonuma-Hecke algebra in the case $\lambda=(n-1,1)$. The Section 7 is devoted to description of the basis for the Hecke algebra $\mathcal{H}(\GL{n},U_{(k,n-k)})$ and some miscellaneous computations related to the Radon transform in that case. In the same section we also formulate a Conjecture \ref{Conjecture} which states that all the possible eigenvalues of the Radon transform are powers of $q$. Furthermore, we prove that in the parabolic case $\lambda=(k,n-k)$ the eigenvalues are also given by integer powers of $q$.\\

\subsection{Acknowledgments} This paper would not be possible without the support of MIT PRIMES-USA program and its organizers Prof.~Pavel Etingof, Dr.~Slava Gerovitch, and Dr.~Tanya Khovanova. We would like to thank Prof.~Roman Bezrukavnikov and  Prof.~Pavel Etingof for many useful discussions. The idea for this project was proposed by Prof.~Roman Bezrukavnikov. The proof of the Proposition \ref{max_parabolic} was suggested to us by ChatGPT 5.4 Pro.  


\section{Preliminaries}
\subsection{Classical Results}

\begin{defn}
    A group algebra $k[G]$, with $k$ a field and a finite group $G$, is the $k$-algebra generated by elements of the group $G$. The multiplication map is given by the following rule:
    \begin{equation*}
        m:g\cdot h \mapsto gh, \quad \forall g,h\in G.
    \end{equation*}
    And the identity element $e\in G$ gives us the unit of this algebra.
\end{defn} 

\begin{defn}
    A left $A$-module for a $k$-algebra $A$ is a $k$-vector space $V$ with a bilinear operation $A \times V \to V$ which satisfies $a(bv) = (ab)v$ for all $a, b \in A,v \in V$ and $1_A\cdot v=v,$ for all $v\in V$.\\
    
    A right $A$-module is defined similarly, and a $(A, B)$-bimodule has an operation of $A$ on the left and $B$ on the right such that both actions commute with each other.
\end{defn}

Left and right $A$-modules are specific cases of bimodules; namely, a left $A$-module is an $(A, k)$-bimodule, and a right $A$-module is the same as a $(k, A)$-bimodule.\\

The group algebra $k[G]$ has a natural structure of a left $k[G \times G]$-module given by
\begin{equation*}
    (g,h)\circ a := gah^{-1} \quad \forall g,h,a \in G.
\end{equation*}
Equivalently, it has a structure of a $(k[G],k[G])-$bimodule given by
\[
    (g,h)\circ a := gah \quad \forall g,h,a \in G.
\]

\begin{defn}\label{convolution}
    A space of functions $Fun(G,k)$ on a finite group is a commutative $k$-algebra of maps
    \begin{equation*}
        f: G \rightarrow k
    \end{equation*}
    with point-wise addition and multiplication. It admits a structure of a left $k[G\times G]$-module with the following action:
    \begin{equation*}
        (g,h)\circ f(a):= f(g^{-1} a h), \quad \forall g,h,a\in G.
    \end{equation*}
    Additionally, the space $Fun(G,k)$ admits another multiplication structure called the convolution product
    \begin{equation*}
        m': f*g(a) :=\sum_{h\in G} f(h)g(h^{-1}a)=\sum_{h\in G} f(ah)g(h^{-1}).
    \end{equation*}
    The unit for $m'$ is given by the following function
    \begin{equation*}
        \delta_e(a):=\begin{cases}
            1, & a=e,\\
            0, & \text{otherwise.}
        \end{cases}
    \end{equation*}
\end{defn}

In this paper we will work with function spaces, so we will need the following proposition.\\

The following Proposition relates convolution product on the space of functions on $G$ and the group algebra $k[G]$ (see \cite[Section 2, Lemma 1]{DB}).

\begin{prop}\label{Functions-group_algebra_correspondence}
    Let us denote by $\phi$ the following map
    \begin{equation}\label{fun_iso}
        \phi: Fun(G,k) \rightarrow k[G], \quad f \mapsto \sum_{g\in G} f(g) g.
    \end{equation}
    Then $\phi$ is an isomorphism of algebras $(Fun(G,k),m')$ and $(k[G],m)$. Furthermore, the algebra isomorphism $\phi$ is an isomorphism of left $k[G \times G]$-modules.
\end{prop}

\begin{defn}
    A representation of an associative algebra A (also called a left $A$-module) is a vector space $V$ equipped with a homomorphism $\rho$ : $A \to \End(V)$, i.e., a linear map preserving the multiplication and unit (see \cite{E}). 
\end{defn}

\begin{defn}
    A \textbf{counit} $\epsilon$ for a group algebra $k[G]$ is a linear map
    \[
    \epsilon: k[G] \rightarrow k, \quad \epsilon(g)=1, \ \forall g\in G.
    \]
\end{defn}

\begin{prop}\label{counit}
    For any two elements $a,b$ from $k[G]$ we have
    \[
    \epsilon(a\cdot b)=\epsilon(a)\epsilon(b).
    \]
\end{prop}

\begin{proof}
    The evaluation of the linear function $\epsilon$ on an element $a$ corresponds to the sum of coefficients of $a$ with respect to the standard basis $g\in G$ of the group algebra. Clearly, multiplication in $k[G]$ preserves the sum of coefficients.
\end{proof}

\begin{defn}
    A \textbf{representation} of a finite group $G$ is a $k[G]$-module $V$. Such a representation is called \textbf{irreducible} iff $V$ does not contain a proper nontrivial $k[G]$-submodule $W\subsetneq V$. A \textbf{character} of $V$ is a linear function $\chi_V$ from $k[G]$ to $k$ given by the formula
    \[
    \chi_V(g)=\tr|_V(g), \quad g\in G.
    \]
    We call $\chi_V$ an \textbf{irreducible character} iff $V$ is an irreducible representation of $G$.
\end{defn}

\begin{defn}\label{idempotents}
    An \textbf{idempotent} (or a projector) $e$ in an algebra $A$ is an element such that $e^2 = e$. Two idempotents $e_1$ and $e_2$ are called \textbf{orthogonal} if $e_1 \cdot e_2 = 0$.
\end{defn}

\begin{ex}\label{subgroup-projector}
    One example of an idempotent corresponds to a subgroup $H$ of a group $G$, defined in the group algebra of $G$ as
    
    \begin{equation*}
        e_H:= \frac{1}{|H|}\sum_{h\in H}h.
    \end{equation*}
    Here we assume that $|H|$ is invertible in $k$.
\end{ex}

We have the following useful Proposition (see \cite[Section 2.4]{FH}).

\begin{prop}\label{e-projector}
    Let $k=\CC$, $G$ be a finite group and let $\psi$ be an irreducible character of $G$, then the element
    \[
    e_{\psi}:=\frac{\psi(e)}{|G|}\sum_{g\in G}\psi(g^{-1})g
    \]
    is an idempotent intertwining $G$-action, where $e$ is the identity element. Moreover, for different irreducible characters $\psi$ and $\phi$ of $G$ the idempotents $e_{\psi}$ and $e_{\phi}$ are orthogonal. We also have
    \[
        \sum_{\psi\in {\rm Irr}(G)}e_{\psi}=1\in \CC[G].
    \]
    Here, the summation is taken over all the irreducible characters of $G$.
\end{prop}



    

The following statement is well known, but we provide a brief proof for it.
\begin{lem}\label{Epsi*U}
    For a one-dimensional character $\psi$ (i.e. $\psi(e)=1$) and $g\in G$ we have $e_\psi \cdot g = g \cdot e_\psi = \psi(g) e_\psi$.
\end{lem}

\begin{proof} The first part follows from
    \begin{equation*}
        e_\psi \cdot g = \frac{1}{|G|} \sum_{h \in G} h \psi(h^{-1})g
        = \frac{1}{|G|} \sum_{h \in G} gh \psi((gh)^{-1}) \psi(g)
        = e_\psi \psi(g).
    \end{equation*}
    The proof for $g \cdot e_\psi$ is identical.
\end{proof}

The following statement follows directly from the Artin–Wedderburn Theorem (e.g.~see \cite[Theorem 4.1.1]{E}) and \cite[Section 6, Proposition 11]{Serre}. 

\begin{thm}\label{AMT} Let $k=\CC$, then as a left $\CC[G\times G]$-module
\begin{equation}\label{Bimod_decomp}
    \CC[G]\cong \bigoplus_{V\in {\rm Irr} (G)} V\otimes V^{*}
\end{equation}
where ${\rm Irr}(G)$ is the set of irreducible representations of $G$. The left $G$-action corresponds to the usual action of $G$ on the first tensor component $V$ and the right $G$-action corresponds to the usual action of $G$ on the second tensor component $V^*$.\\
\end{thm}

We also recall the following definition.\\

\begin{defn}
    Let $H\subset G$ be two finite groups and let $(V,\rho)$ be a representation of the group $G$. The restricted representation ${\rm Res}_H^{G} \rho$ is a representation of the subgroup $H$ on the vector space $V$ with the action given by
    \begin{equation*}
        h\circ v:=\rho(h)v, \quad \forall h\in H, v\in V.
    \end{equation*}
\end{defn}

Throughout this paper we will sometimes use restrictions of module structures to subalgebras.

\subsection{Bruhat Decomposition, Roots and Permutations}


Let $GL_n(\mathbb{F}_q)$ be a general linear group over the finite field $\mathbb{F}_q$. For a more detailed account on the subject of this subsection see \cite{DB}.

\begin{defn}
    We define the \textbf{Borel subgroup} $B_n \subset \GL{n}$ to be the subgroup of upper triangular matrices.
\end{defn}
Then the following Theorem is well known.

\begin{thm}\label{bruhat}
    Every matrix in $GL_n (\mathbb{F}_q)$ can be expressed as $b_1wb_2$, where $b_1$ and $b_2$ are elements of $B_n$, the Borel subgroup, and $w\in S_n \subset \GL{n}$ is a unique permutation matrix, i.e. an invertible matrix with a $1$ in every row and column and $0$ elsewhere.
\end{thm}\label{unipotent_Bruhat}
In fact, if $U_{\lambda_n}\subset \GL{n},\lambda_n:=1^{(n)}$ is a subgroup of upper-triangular matrices with $1$s on the diagonal and $T_n\subset \GL{n}$ is a subgroup of diagonal matrices, it is not difficult to show by row and column elimination a stronger version of this Theorem (or see, for example, \cite[Theorem 4]{St}).

\begin{thm}
    Every matrix in $GL_n (\mathbb{F}_q)$ can be expressed as $u_1twu_2$, where $u_1$ and $u_2$ are elements of $U_{\lambda_n}$, $w\in S_n \subset \GL{n}$ is a unique permutation matrix and $t$ is a unique element of $T_n$.
\end{thm}
We will need a few notions and facts regarding root systems.
\begin{defn}
    The \textbf{root system} $\Phi$ for $\GL{n}$ is a subset of $\mathbb{R}^n$ with a basis given by $\epsilon_i,1\le i\le n$ such that
    \[
    \Phi=\{\epsilon_i-\epsilon_j\in\mathbb{R}^n|1\le i\neq j \le n\}.
    \]
    We call the complementary subsets of $\Phi$
    \[
    \Phi^+=\{\epsilon_i-\epsilon_j\in\mathbb{R}^n|1\le i< j \le n\},\quad \Phi^-=\{\epsilon_i-\epsilon_j\in\mathbb{R}^n|1\le j< i \le n\}
    \]
    as \textbf{positive} and \textbf{negative} roots respectively. The subset
    \[
    \Delta=\{\alpha_i:=\epsilon_i-\epsilon_{i+1}|1\le i<n\}
    \]
    of $\Phi^+$ is called \textbf{simple} roots.
\end{defn}



The group $S_n$ acts on the set $\Phi$ via permutations on indices of $\epsilon_i$. It also acts on $T_n$ by conjugation. Let $s_i:=(i,i+1)\in S_n, 1\le i<n$ be the elementary transposition. For an element $w \in S_n$, we define the sets
    \[
    \Phi^-_w = \{\alpha \in \Phi^+ |\ w\cdot \alpha \in \Phi^-\}\ \text{ and }\ \Phi^+_w:=\Phi^+ \setminus \Phi^-_w.
    \]
    The following Theorem holds true (see \cite{DB,H}). 

\begin{thm}\label{reduced_decomposition}
    The symmetric group $S_n$ is generated by the elements $s_i$ satisfying the Coxeter relations
    \begin{equation}\label{Coxeter}
        s_is_j=s_js_i, \ |i-j|>1, \quad s_is_{i+1}s_i=s_{i+1}s_is_{i+1},\ 1\le i<n-1, \quad s_i^2=e.
    \end{equation}
    Any element $w\in S_n$ admits a reduced decomposition in terms of $s_i$ of minimal length $l(w)$. Any two reduced decompositions of $w$ can be obtained from each other by applying the first two relations from \eqref{Coxeter}, which are called the braid relations.\\
    
    Additionally, $l(ws_i) = l(w)+1$ iff $\alpha_i \in \Phi_w^+$ and $l(ws_i) = l(w)-1$ iff $\alpha_i \in \Phi_w^-$.
\end{thm}
\subsection{Radon Transform}

\begin{defn}
    Let us fix a tuple of positive integers $\lambda=(\lambda_1,\dots, \lambda_l)$ partitioning $n$. Let $M_{\lambda}$ denote the subgroup of block-diagonal matrices $\prod_{i=1}^l GL_{\lambda_i}(\mathbb{F}_q)$ of sizes $\lambda$ in $GL_{n}(\mathbb{F}_q)$. We denote by $U_{\lambda}$ the subgroup of unipotent upper-triangular matrices with nontrivial entries above the blocks of $M_{\lambda}$. Then $U_{\lambda}^{op}$ is the subgroup of transposed matrices from $U_{\lambda}$. Finally, we denote by $P_{\lambda}$ ($P_{\lambda}^{op}$) the subgroup corresponding to the product of subgroups $M_{\lambda}\cdot U_{\lambda}=U_{\lambda} \cdot M_{\lambda}$ ($M_{\lambda}\cdot U_{\lambda}^{op}$ correspondingly). We call $P_{\lambda}$ a parabolic subgroup of type $\lambda$, $M_{\lambda}$ is the Levi subgroup of $P_{\lambda}$ and $U_{\lambda}$ is the unipotent radical of $P_{\lambda}$ (and the same for the opposite subgroups).
\end{defn}

\begin{ex}\label{PLU_examples}
    Let $\lambda=(2,3,1)$, then we can visualize $P_{\lambda},M_{\lambda},U_{\lambda}$ as follows
    \begin{equation*}
        P_{\lambda}=
        \begin{pmatrix}
            * & * & * & * & * & *\\
            * & * & * & * & * & *\\
            0 & 0 & * & * & * & *\\
            0 & 0 & * & * & * & *\\
            0 & 0 & * & * & * & *\\
            0 & 0 & 0 & 0 & 0 & *\\
        \end{pmatrix}
    \end{equation*}
    \begin{equation*}
        M_{\lambda}=
        \begin{pmatrix}
            * & * & 0 & 0 & 0 & 0\\
            * & * & 0 & 0 & 0 & 0\\
            0 & 0 & * & * & * & 0\\
            0 & 0 & * & * & * & 0\\
            0 & 0 & * & * & * & 0\\
            0 & 0 & 0 & 0 & 0 & *\\
        \end{pmatrix}
    \end{equation*}
    \begin{equation*}
        U_{\lambda}=
        \begin{pmatrix}
            1 & 0 & * & * & * & *\\
            0 & 1 & * & * & * & *\\
            0 & 0 & 1 & 0 & 0 & *\\
            0 & 0 & 0 & 1 & 0 & *\\
            0 & 0 & 0 & 0 & 1 & *\\
            0 & 0 & 0 & 0 & 0 & 1\\
        \end{pmatrix}
    \end{equation*}
\end{ex}

Note that $U_{\lambda}$ is a normal subgroup in $P_{\lambda}=M_{\lambda}\ltimes U_{\lambda}$.\\

\begin{defn}
    Let $H\subset G$ be two finite groups. For a field $k$ we denote by $Fun(G/H,k)$ ($Fun(H\backslash G,k)$ or $Fun(H\backslash G /H,k)$) the subspace of all the functions from $Fun(G,k)$ invariant under the right (respectively left or left and right) $H$-action.\\
\end{defn}

\begin{defn}
    In the same setting we denote by $k[G/H]$ ($k[H\backslash G]$ or $k[H\backslash G/H]$) the subspace of all elements from $k[G]$ invariant under the right (respectively left or left and right) $H$-multiplication. For a coset $gH\in G/H$ ($Hg$ or $HgH$) we associate an element $gH$ in $k[G/H]$ (similar for other cases) given by
    \begin{equation*}
        gH:=\sum_{h\in H} gh.
    \end{equation*}
    Note that the element $gH$ does not depend on the choice of a representative $g$.
\end{defn}

We also recall a general notion of a Hecke algebra which will be useful later.

\begin{defn}\label{finite_group_Hecke_algebra}
    A Hecke algebra $\mathcal{H}(G,H)$ for the pair of finite groups $H\subset G$ is an algebra of $H$-bi-invariant functions $Fun(H\backslash G/H)$ on $G$ with multiplication given by the convolution and a unit $1_H/|H|$ where $1_H$ is the function equal to identity on the double coset of the unit $e\in G$ and zero everywhere else.
\end{defn}

\begin{rem}\label{convol_Hecke_alg}
    By using the map $\phi$ from the Proposition \ref{Functions-group_algebra_correspondence} we may think of $\mathcal{H}(G,H)$ as follows.
    Under the map $\phi$ the Hecke algebra $\mathcal{H}(G,H)$ can be identified with the subalgebra $e_H k[G] e_H$ of $k[G]$ with a unit given by $e_H$ and the multiplication given by
    \begin{equation*}
        e_H g_1 e_H \cdot e_H g_2 e_H = \frac{1}{|H|} \sum_{h\in H} e_H g_1hg_2 e_H .
    \end{equation*}
\end{rem}
Due to Proposition \ref{Functions-group_algebra_correspondence} we have the following.\\

\begin{prop}\label{convol_Hecke_mod}
    Let $H\subset G$ be two finite groups. Then $\phi$ gives rise to the isomorphism between $Fun(G/H,k)  , Fun(H\backslash G,k) ,  Fun(H\backslash G/H,k)$ and $k[G/H]  , k[H\backslash G] ,  k[H\backslash G/H]$ respectively. The elements $g$ for corresponding cosets form a basis for the respective spaces. Moreover, the restrictions of $\phi$ are isomorphisms of left $k[G \times N_G(H)]  , k[N_G(H)\times G] ,$ $k[N_G(H)\times N_G(H)]$-modules respectively. Here $N_G(H)$ denotes the normalizer of $H$ in $G$.
\end{prop}

\begin{defn}[{\bf Radon transform}]\label{Radon_transform}
    Consider the following diagram. 
    \begin{equation*}
        \begin{tikzcd}
            & {Fun(GL_n(\mathbb{F}_q),\mathbb{C})} \arrow[rd, "\pi_*"', shift right] \arrow[ld, "\pi^{op}_*"', shift right] & \\
            {Fun(GL_n(\mathbb{F}_q)/U_{\lambda}^{op},\mathbb{C})} \arrow[ru, "\pi^{op *}"', shift right] &                                                                                                               & {Fun(GL_n(\mathbb{F}_q)/U_{\lambda},\mathbb{C})} \arrow[lu, "\pi^*"', shift right]
        \end{tikzcd}
.
    \end{equation*}
    Where the maps $\pi_*,\pi^*,\pi_*^{op},\pi^{op*}$ are given by the "pull back" and "push forward" operators. For instance, $\pi^*$ is the operator that is defined by $\pi^*(f)(g) = f(gU_\lambda)$, and $\pi_*$ is the operator $\pi_*(f)(gU_\lambda) = \sum_{u \in U_\lambda} f(gu)$. Using this notation, we denote $R = \pi^{*} \pi^{op}_*$ and $R' = \pi^{op*} \pi_*$. We will study the Radon operator which we will denote as $\R:=R' \circ R$.
\end{defn}


\section{Properties of the Radon Transform}

Here we explore some of the properties of the Radon transform. We make an observation below.\\


\begin{prop}\label{Radon Expansion}
    Under the map $\phi$ from the Proposition \ref{Functions-group_algebra_correspondence} the Radon transform $\R$ from Definition \ref{Radon_transform} corresponds to the endomorphism of the subspace $\CC[\GL{n}/U_{\lambda}]$ of right $U_{\lambda}$-invariants in $\CC[\GL{n}]$. Namely, it is given by the operation of the right multiplication
    \begin{equation}
        R_{\R}: \CC[\GL{n}]e_{U_{\lambda}} \rightarrow \CC[\GL{n}]e_{U_{\lambda}}, \quad g\cdot e_{U_{\lambda}} \mapsto g\cdot |U_{\lambda}|^2e_{U_{\lambda}}e_{U_{\lambda}^{op}}e_{U_{\lambda}}.
    \end{equation}
    From now on we work with the group element $\R:=|U_{\lambda}|^2e_{U_{\lambda}}e_{U_{\lambda}^{op}}e_{U_{\lambda}} \in \CC[\GL{n}]$.
    

\end{prop}

The classification of nonzero eigenvalues of $\R$ acting on $\CC[\GL{n}]e_{U_{\lambda}}$ is equivalent to classification of eigenvalues of $\R$ on $\CC[\GL{n}]$ since $\R$ projects the group algebra onto the subspace of right $U_{\lambda}$-invariants.\\

Note that $\CC[\GL{n}]e_{U_{\lambda}}$ admits a structure of a left $\GL{n}$ module and a right $M_{\lambda}$-module by multiplication on the left and on the right correspondingly, as $m e_{U_{\lambda}}=e_{U_{\lambda}} m$, $\forall m\in M_{\lambda}$.

\begin{lem}\label{intertwine}
    The Radon transform intertwines left $GL_n(\mathbb{F}_q)$ and right $M_{\lambda}$ actions. In particular, any of its eigenvalues comes from the action of $\R$ on an irreducible representation $V$ of $GL_n(\mathbb{F}_q)$.
\end{lem}
\begin{proof}
    We know that Radon transform acts on $\mathbb{C}[GL_n(\mathbb{F}_q)]$ by right multiplication by $|U_{\lambda}|^2 e_{U_{\lambda}}e_{U_{\lambda}^{op}}e_{U_{\lambda}}$, so it commutes with left multiplication and right $M_{\lambda}$ action (as $U_{\lambda}$ and $U_{\lambda}^{op}$ are normal subgroups of $P_{\lambda}$ and $P_{\lambda}^{op}$ correspondingly). It follows then from \thmref{AMT} that $\R$ acts on second tensor factors in \eqref{Bimod_decomp}.\\
\end{proof}

The following proposition allows us to obtain some information about Radon transform on $\mathbb{C}[GL_n(\mathbb{F}_q)/U_{\lambda}]$ by considering $P_{\lambda}$-invariant functions.\\

\begin{prop}\label{parabolic-unipotent}
    Consider the following diagram
    \begin{equation}\label{parabolic-unipotent_maps}
        \begin{tikzcd}
{Fun(GL_n(\mathbb{F}_q)/P_{\lambda},\mathbb{C})} \arrow[rr, "i", hook, shift left] &  & {Fun(GL_n(\mathbb{F}_q)/U_{\lambda},\mathbb{C})} \arrow[ll, "j", two heads, shift left]
\end{tikzcd}
    \end{equation}
    where $i$ is the tautological inclusion and $j$ is given by the fiber integration with respect to the quotient map $\mu: GL_n(\mathbb{F}_q)/U_{\lambda} \twoheadrightarrow GL_n(\mathbb{F}_q)/P_{\lambda}$. Then $i,j$ commute with $\R$ and we have $j\circ i = |M_{\lambda}| \cdot id$.
\end{prop}
\begin{proof}
    The last statement follows by observing that 
    \begin{equation*}
        \mu^{-1}(gP_{\lambda})=(g\cdot M_{\lambda}) U_{\lambda} \quad \forall gP_{\lambda}\in GL_n(\mathbb{F}_q)/P_{\lambda}.
    \end{equation*}
    The commutativity of $i$ and $\R$ follows from consideration of the isomorphism $\phi$ from \eqref{fun_iso}. We deduce the statement for $j$ by observing that $\phi\circ j \circ \phi^{-1}$ is given by the right multiplication by $|M_{\lambda}|e_{M_{\lambda}}$. We note that $e_{M_{\lambda}}e_{U_{\lambda}}=e_{U_{\lambda}}e_{M_{\lambda}}$ and $e_{M_{\lambda}}e_{U_{\lambda}^{op}}=e_{U_{\lambda}^{op}}e_{M_{\lambda}}$.\\
\end{proof}

In particular, we see that any eigenvalue of Radon transform appearing in \linebreak $Fun(GL_n(\mathbb{F}_q)/P_{\lambda},\mathbb{C})$ appears in $Fun(GL_n(\mathbb{F}_q)/U_{\lambda},\mathbb{C})$ as well. This is because an eigenvector in $Fun(GL_n(\mathbb{F}_q)/P_{\lambda},\mathbb{C})$ can be embedded into $Fun(GL_n(\mathbb{F}_q)/U_{\lambda},\mathbb{C})$ via $i$.\\


\begin{ex}
    The Proposition \ref{parabolic-unipotent} allows us to find the eigenvalues of the Radon transform in the case of $\lambda = (1, n-1)$ corresponding to the eigenvectors lying in the image of $i$ from \eqref{parabolic-unipotent_maps}. Namely, any such eigenvalue is either $q^{n-2}$ or $q^{2n-2}$.\\
\end{ex}

\begin{proof}
    Note that $GL_n(\mathbb{F}_q) / P_\lambda$ is relatively small, and the coset elements can be enumerated by points on the projective space $\mathbb{P}^{n-1}$. Recall that $P_\lambda$ is the set of invertible matrix of the following form
    \begin{equation*}
        \begin{pmatrix}
            * & * & \cdots & * \\
            0 & * & \cdots & *\\
            \vdots & \vdots & & \vdots \\
            0 & * & \cdots & *\\
        \end{pmatrix}.
    \end{equation*}

    Let $GL_n(\mathbb{F}_q)$ act on a vector space $V$ of dimension $n$ over $\mathbb{F}_q$. Let us fix a basis $b_1, \dots, b_n$ of $V$.\\

    Note that the image of $b_1$ under the action of any representative of the coset $gP_\lambda\in \GL{n}/P_{\lambda}$ for a $g \in \GL{n}$ is a vector proportional to the first column in $g$. In addition, all other basis elements $b_2,\dots,b_n$ can be sent to any collection of vectors in $V$ linearly independent from $g$ (and each other) for an appropriate choice of a representative of the coset $gP_\lambda$. Thus, we may identify the coset with the projective space $\mathbb{P}(V)$ of lines in $V$.\\

    Now we wish to compute the matrix of the action of the Radon transform on the coset space $\CC[\GL{n}/P_{\lambda}]=\CC[\mathbb{P}(V)]$. Recall that the action is given by
    \begin{equation}\label{example_mirabolic}
        \R \cdot gP_{\lambda}= \sum_{u \in U_\lambda}\sum_{u^{op} \in U_\lambda^{op}}g u  u^{op} P_\lambda.
    \end{equation}

    For a particular choice of $u$ and $u^{op}$, we need to find the image of the line $\langle b_1 \rangle$ corresponding to the first basis element $b_1$. 
    After application of the operator $\sum_{u^{op} \in U^{op}}u^{op}$ to $\langle b_1 \rangle$, we get
    
    \begin{equation*}
        \sum_{u^{op} \in U^{op}}u^{op} \cdot \langle b_1 \rangle=\sum_{c_i \in \mathbb{F}_q}\langle b_1 + c_2 b_2 + c_3 b_3 + \cdots + c_n b_n\rangle.
    \end{equation*}

    Now, after applying $\sum_{u \in U_\lambda}u$ we get 
    
    \begin{equation*}
        \sum_{c_i, d_i \in \mathbb{F}_q} \langle(1+c_2d_2+c_3d_3 + \cdots + c_nd_n)b_1+c_2b_2 + c_3b_3 + \cdots +c_n b_n\rangle.
    \end{equation*}
    
    If $c_i = 0$ for all $i$, the coefficient in front of $b_1$ is $1$. We can simplify the summation to
    
    \begin{equation*}
        \sum_{d_i \in \mathbb{F}_q} \langle b_1 \rangle = q^{n-1} \langle b_1 \rangle.
    \end{equation*}
    
    Otherwise, the coefficient in front of $b_1$ varies across the entire set $\mathbb{F}_q$ as we have at least one nonzero $c_i$ for some $i$, and by varying $d_i \in \mathbb{F}_q$ we can vary the coefficient in front of $b_1$ freely. So, we can replace it with $d = 1+c_2d_2+c_3d_3 + \cdots + c_nd_n$. Then, the sum corresponding to such terms can be written as
    
    \begin{equation}\label{lines_avoiding_b1}
        q^{n-2}\sum_{\substack{c_i, d \in \mathbb{F}_q\\\text{ not all $c_i$ are }0}} \langle db_1+c_2b_2 + c_3b_3 + \cdots +c_n b_n \rangle.
    \end{equation}

    Now, note that this expression covers all points not on the line $\langle b_1 \rangle$ (and thus all other lines) equally. The number of non-zero points on each line is $q-1$, so the sum \eqref{lines_avoiding_b1} becomes
    \begin{equation*}
         q^{n-2}(q-1) \sum_{\text{all lines } \ell \ne \langle b_1 \rangle} \ell.
    \end{equation*}


    This shows that the matrix of the action of the Radon transform on $\CC[\mathbb{P}(V)]$ is given by $q^{n-2}Id + q^{n-2}(q-1)M$, where $Id$ is the identity matrix and $M$ is the matrix with all entries equal to $1$.\\

    The size of $M$ is the number of lines, which is equal to $\frac{q^n-1}{q-1}$. Then, since ${\rm rk}(M)=1$, the eigenvalues of $M$ are $\frac{q^n-1}{q-1}$ with multiplicity $1$ and $0$ with multiplicity $\frac{q^n-1}{q-1}-1$.\\

    Finally, the eigenvalues of $\R$ are $q^{n-2} + q^{n-2}(q-1)\frac{q^n-1}{q-1}=q^{2n-2}$ and $q^{n-2}+0=q^{n-2}$ with multiplicities $1$ and $\frac{q^n-1}{q-1}-1$ respectively.
\end{proof}

\begin{prop}\label{symmetric_Rd}
    The matrix of the Radon transform with respect to the standard basis is symmetric.
\end{prop}
\begin{proof}
    This is equivalent to showing that the coefficient of $g_2$ in $\R (g_1)$ is the same as the coefficient of $g_1$ in $\R (g_2)$. Fix a representative $g_1' \in \GL{n}$ in $g_1 U_\lambda$ and a representative $g_2' \in \GL{n}$ in $g_2 U_\lambda$. The coefficient of $g_2$ is precisely the number of triples $(u_1, u^{op}, u_2)$, where $u_1, u_2 \in U_\lambda$ and $u^{op} \in U_\lambda^{op}$ that satisfy the equation $g_2 = g_1 u_1 u^{op} u_2$. However, since $U_\lambda$ and $U_\lambda^{op}$ are groups, we can take the inverse, and $g_1 = g_2 u_2^{-1} (u^{op})^{-1}u_1^{-1}$. In this way, we have a bijection between the triples that correspond to the coefficient of $g_2$ and the triples that correspond to the coefficient of $g_1$, so the number of these are equal and the matrix of the Radon transform with respect to the standard basis is symmetric.\\
\end{proof}

\begin{cor}
    All the eigenvalues of the Radon transform are real.\\
\end{cor} 

\begin{prop}
    The maximum eigenvalue of the Radon transform is $|U_\lambda|^2$.\\
\end{prop}
\begin{proof}
    Note that the matrix of the Radon transform $\R$ in the standard basis is given by non-negative numbers. Additionally, any row sum of such a matrix is equal to $|U_{\lambda}||U_{\lambda}^{op}|=|U_{\lambda}|^2$. Therefore, the Perron–Frobenius eigenvalue of $\R$ is not greater than $|U_{\lambda}|^2$. On the other hand, such an eigenvalue can be achieved by substituting all $1$s for vector coefficients.
\end{proof}

\begin{prop}\label{blocks}
    The matrix of the Radon transform can be divided into $q-1$ equal sized and equivalent blocks.\\
\end{prop}
\begin{proof}
    This is because we can consider the determinant of the matrix. Since $U_\lambda$ and $U_\lambda^{op}$ only contain determinant $1$ matrices, it acts separately on matrices of each determinant. It acts equivalently on each determinant matrix because we can simply factor out a matrix of the same determinant (i.e. determinant $2$ from a determinant $2$ matrix) using \ref{intertwine} to show that the action is symmetric across all determinant classes. Since there are $q-1$ different determinants, the matrix can be divided into $q-1$ equal sized and equivalent blocks.\\
\end{proof}

\begin{cor}
    All the multiplicities of eigenvalues of the Radon transform are divisible by $q-1$.\\
\end{cor}

Note that in fact, the Radon transform can be defined for the group algebra $A[\GL{n}]$ over any commutative ring $A$ such that $q$ is invertible in $A$.\\

\begin{prop}
    Let $A$ be a commutative ring in which $q$ is invertible. The Radon transform $\R$ is an isomorphism of left $A[\GL{n}\times M_{\lambda}]$-modules
    \[
    \R: A[\GL{n}/U_{\lambda}]\xrightarrow{\sim} A[\GL{n}/U_{\lambda}].
    \]
    Here $A[\GL{n}/U_{\lambda}]=A[\GL{n}]e_{U_{\lambda}}$ (since $q$ is invertible).
\end{prop}

\begin{proof}
    This statement follows directly by applying \cite[Theorem 2.4]{HL} for $J=K,\ w=w_0w_J$ in the original notation. To be more precise, let $w_0$ be the longest permutation in $S_n$ and $w_J$ be the longest permutation from the subgroup of permutation matrices in $M_{\lambda}$. Then \cite[Theorem 2.4]{HL} states that the linear map
    \[
    T: A[\GL{n}/U_{\lambda}] \rightarrow A[\GL{n}/U_{\lambda}], \quad T(x)=xe_{U_\lambda}w_0w_Je_{U_\lambda}
    \]
    is an isomorphism. Observe that $w_0w_J U_{\lambda}(w_0w_J)^{-1}=U_{\lambda}^{op}$. Therefore, if we compose $T$ with the conjugation by $(w_0w_J)^{-1}$, we will see that the map
    \[
    R: A[\GL{n}/U_{\lambda}] \rightarrow A[\GL{n}/U_{\lambda}^{op}], \quad R(x)=(w_0w_J)^{-1} T(x)w_0w_J
    \]
    is an isomorphism. It follows that $\R$ is invertible as well.
\end{proof}

\begin{cor}
    The eigenvalues of $\R$ are nonzero algebraic integers dividing powers of $q$.
\end{cor}

The following theorem helps us to describe the set of possible eigenvalues for the Radon transform with general $\lambda$ via the special case $\lambda=(k,n-k)$.\\

\begin{thm}\label{eigenvalues_box}
    Let $\lambda=(\lambda_1,\dots,\lambda_l)$ be a tuple of positive integers partitioning $n$ and $\lambda'=(\lambda_1,\dots,\lambda_{l-1})$. Suppose that $S=\{a_1,\dots, a_r\}$ and $S'=\{b_1,\dots, b_p\}$ are the sets of possible eigenvalues for corresponding Radon transforms of $(GL_{n}(\mathbb{F}_q),(n-\lambda_l,\lambda_l))$ and $(GL_{|\lambda'|}(\mathbb{F}_q), \lambda')$. Then the set of possible eigenvalues for Radon transform in the case of $(GL_{n}(\mathbb{F}_q),\lambda)$ is a subset of $S\cdot S':=\{a_1b_1, a_1b_2,\dots, a_rb_p\}$.
\end{thm}

\begin{proof}
    Let $U_{(n-\lambda_l,\lambda_l)}$ be the unipotent radical in $GL_{n}(\mathbb{F}_q)$ of type $(n-\lambda_l,\lambda_l)$ and $e_1=e_{U_{(n-\lambda_l,\lambda_l)}},e_1^{op}:=e_{U_{(n-\lambda_l,\lambda_l)}^{op}}$ be the corresponding projectors of the same type. Let $GL_{|\lambda'|}(\mathbb{F}_q)$ be the subgroup of $GL_n(\mathbb{F}_q)$ embedded into the upper left corner and let $e_2:=e_{U_{\lambda'}},e_2^{op}:=e_{U_{\lambda'}^{op}}$ be the corresponding subgroup projectors corresponding to $U_{\lambda'},U_{\lambda'}^{op}\subset GL_{|\lambda'|}(\mathbb{F}_q) \subset GL_{n}(\mathbb{F}_q)$.\\

    Since $U_{\lambda'}\cdot U_{(n-\lambda_l,\lambda_l)}=U_{\lambda}$ and $U_{\lambda'}\subset M_{\lambda}$, we know that $e_1,e_1^{op}$ and $e_2,e_2^{op}$ commute and
    \begin{equation}
        e_1 e_2=e_{U_{\lambda}}, \quad e_1^{op} e_2^{op}=e_{U_{\lambda}^{op}}.
    \end{equation}
    Therefore the corresponding smaller Radon transforms commute and their product gives us the larger Radon transform.\\

    Now we note that for any irreducible representation $V$ of $GL_n(\mathbb{F}_q)$ its space $V^{U_{\lambda}}$ of $U_{\lambda}$-invariants is $U_{(n-\lambda_l,\lambda_l)}$ and $U_{\lambda'}$-invariant. Since both transforms are diagonalizable and they commute with each other, both transforms can be diagonalized simultaneously, and the statement follows.\\
\end{proof}



\section{The Case of $\GL{2}$}


From our results from the section above, we can figure out the calculations for the $\GL{2}$ case.

\begin{thm}\label{eigenvalues_gl2}
    The only possible eigenvalues of the Radon transform on $\GL{2}$ are $1$, $q$, and $q^2$.
\end{thm}

\begin{proof}
Let $\lambda=(1,1)$ and $e_{U_{\lambda}}=\frac{1}{|U_{\lambda}|}\sum_{u \in U_\lambda} u$ as defined in \ref{subgroup-projector}. Recall from the Definition \ref{Radon Expansion} that $\R =q^2 e_{U_{\lambda}}e_{U_{\lambda}^{op}}e_{U_{\lambda}}$.\\


We explicitly write out the minimal polynomial of $\R$. For this we will find similarities between the elements $q^3 e_{U_{\lambda}}e_{U_{\lambda}^{op}}e_{U_{\lambda}}$ and $q^3 e_{U_{\lambda}^{op}}e_{U_{\lambda}}e_{U_{\lambda}^{op}}$. We have

\[q^3 e_{U_{\lambda}}e_{U_{\lambda}^{op}}e_{U_{\lambda}} = \sum_{a, b, c \in \mathbb{F}_q} 
\begin{pmatrix}
1 & a \\
0 & 1 
\end{pmatrix}
\begin{pmatrix}
1 & 0 \\
b & 1 
\end{pmatrix}
\begin{pmatrix}
1 & c \\
0 & 1 
\end{pmatrix}
= \sum_{a, b, c \in \mathbb{F}_q} 
\begin{pmatrix}
1+ab & c+a+abc \\
b & bc+1 
\end{pmatrix}.\]

Similarly,

\[q^3 e_{U_{\lambda}^{op}}e_{U_{\lambda}}e_{U_{\lambda}^{op}} = \sum_{a', b', c' \in \mathbb{F}_q} 
\begin{pmatrix}
1 & 0 \\
a' & 1 
\end{pmatrix}
\begin{pmatrix}
1 & b' \\
0 & 1 
\end{pmatrix}
\begin{pmatrix}
1 & 0 \\
c' & 1 
\end{pmatrix}
= \sum_{a', b', c' \in \mathbb{F}_q} 
\begin{pmatrix}
1+b'c' & b' \\
c'+a'+a'b'c' & a'b'+1 
\end{pmatrix}.\]

Setting both matrices to be equal, we get the system of equations:
\[
    b'c' = ab, \quad a'b' = bc, \quad b' = a+c+abc, \quad b = a'+c'+a'b'c'.
\]

Using the first three equations, we get the following solutions
\[
    a' = \frac{bc}{a+c+abc}, \quad b' = a+c+abc, \quad c' = \frac{ab}{a+c+abc}.
\]
Here, solutions exist only if $a+c+abc \ne 0$. Note that this corresponds to the matrices where $b' = 0$. This makes sense as if $b' = 0$, then the corresponding summands of $\R$ are lying exclusively in $U^{op}_\lambda$. There are multiple choices for $a'$ and $c'$ for every element, so it makes sense that we cannot find a single solution for $a'$ $b'$ and $c'$. In addition, we also need to account for the matrices where $b=0$, or when $c'+a'+a'b'c' = 0$. This is again a problem as there are multiple values of $a$ and $c$ that satisfy this equality, so we need to remove these matrices before comparing the two expressions.\\

The matrices where $b' = 0$ are equivalent to the matrices where $a+c+abc = 0$. This can be split into two cases, one where $a$ and $c$ are both $0$, and the other where $b = -\frac{a+c}{ac}$.\\

The other case where $b = 0$ is equivalent. Then, we also have to make sure we do not double-count the cases where both $b = 0$ and $b' = 0$, which is simply just the cases where the matrix is the identity.\\

From the observations above we get that

\[q^3 e_{U_{\lambda}}e_{U_{\lambda}^{op}}e_{U_{\lambda}} - \sum_{a, c \in \mathbb{F}_q} 
\begin{pmatrix}
1 & a+c \\
0 & 1 
\end{pmatrix} - \sum_{b \in \mathbb{F}_q}
\begin{pmatrix}
1 & 0 \\
b & 1 
\end{pmatrix}
- \sum_{a, c \in \mathbb{F}_q^{\times}}
\begin{pmatrix}
-a/c & 0 \\
-(a+c)/ac & -c/a 
\end{pmatrix}
+\sum_{a \in \mathbb{F}_q}
\begin{pmatrix}
1 & 0 \\
0 & 1 
\end{pmatrix}\]

is equal to

\[q^3 e_{U_{\lambda}^{op}}e_{U_{\lambda}}e_{U_{\lambda}^{op}} - \sum_{a', c' \in \mathbb{F}_q} 
\begin{pmatrix}
1 & 0 \\
a'+c' & 1 
\end{pmatrix} - \sum_{b' \in \mathbb{F}_q}
\begin{pmatrix}
1 & b' \\
0 & 1 
\end{pmatrix}
- \sum_{a', c' \in \mathbb{F}_q^{\times}}
\begin{pmatrix}
-c'/a' & -(a'+c')/a'c' \\
0 & -a'/c' 
\end{pmatrix}
+\sum_{a' \in \mathbb{F}_q}
\begin{pmatrix}
1 & 0 \\
0 & 1 
\end{pmatrix}.\]

Simplifying similar terms and remembering that $\sum_{a \in \mathbb{F}_q} \begin{pmatrix}
1 & a \\
0 & 1 
\end{pmatrix} = qe_{U_{\lambda}}$, we get the following relation
\begin{multline*}
    q^3 e_{U_{\lambda}}e_{U_{\lambda}^{op}}e_{U_{\lambda}} = q^3 e_{U_{\lambda}^{op}}e_{U_{\lambda}}e_{U_{\lambda}^{op}} + (q-1) qe_{U_{\lambda}} - (q-1) qe_{U_{\lambda}^{op}} +\\
    +
    \sum_{a, c \in \mathbb{F}_q^{\times}}\begin{pmatrix}
    -a/c & 0 \\
    -(a+c)/ac & -c/a 
    \end{pmatrix}
    - \sum_{a', c' \in \mathbb{F}_q^{\times}}
    \begin{pmatrix}
    -c'/a' & -(a'+c')/a'c' \\
    0 & -a'/c' 
    \end{pmatrix}.
\end{multline*}

Now, we can write $\R^2$ as $ q^4 e_{U_{\lambda}}e_{U_{\lambda}^{op}}e_{U_{\lambda}}e_{U_{\lambda}}e_{U_{\lambda}^{op}}e_{U_{\lambda}}=q^4 e_{U_{\lambda}}e_{U_{\lambda}^{op}}e_{U_{\lambda}}e_{U_{\lambda}^{op}}e_{U_{\lambda}}$. Noting that $e_{U_{\lambda}}e_{U_{\lambda}^{op}}e_{U_{\lambda}}e_{U_{\lambda}^{op}}e_{U_{\lambda}} = e_{U_{\lambda}}(e_{U_{\lambda}^{op}}e_{U_{\lambda}}e_{U_{\lambda}^{op}})e_{U_{\lambda}}$, we can expand and get

\begin{multline*}
    \R^2 =  q^4e_{U_{\lambda}}e_{U_{\lambda}^{op}}e_{U_{\lambda}} -  (q-1) q^2e_{U_{\lambda}} +  (q-1) q^2e_{U_{\lambda}}e_{U_{\lambda}^{op}}e_{U_{\lambda}}-\\
    - qe_{U_{\lambda}} \sum_{a, c \in \mathbb{F}_q^{\times}}\begin{pmatrix}
    -a/c & 0 \\
    -(a+c)/ac & -c/a 
    \end{pmatrix} e_{U_{\lambda}}
    + qe_{U_{\lambda}} \sum_{a', c' \in \mathbb{F}_q^{\times}}
    \begin{pmatrix}
    -c'/a' & -(a'+c')/a'c' \\
    0 & -a'/c' 
    \end{pmatrix} e_{U_{\lambda}}.
\end{multline*}

Let's consider the first sum, the second sum is essentially identical. Fix a value of $a$ and $c$.\\

Now, consider the pair $ka, kc$. The matrix then becomes 
\begin{equation*}
    \begin{pmatrix}
-ka/kc & 0 \\
-(ka+kc)/kakc & -kc/ka 
\end{pmatrix} = \begin{pmatrix}
-a/c & 0 \\
-(a+c)/kac & -c/a 
\end{pmatrix}.    
\end{equation*}

Note that this is simply multiplying the bottom left element by $k^{-1}$. Thus, if $a+c \ne 0$, we can always make any value appear in the bottom left square. In addition, by varying $a$, we can obtain any value in the top left (other than $0$) with equal frequency. However, note that if $a+c = 0$, then the element on the diagonal is $1$. Thus, we can characterize all possible matrices in the following way: the element on the top left runs over all values in $\mathbb{F}_q^{\times}$ with equal frequency, the element in the bottom right is its inverse, and the element in the bottom left is in $\mathbb{F}_q^{\times}$ with equal frequency as long as $-a/c \ne 1$, and it is always $0$ otherwise. It is simple to check that all possible matrices show up exactly once with the exception of the identity matrix, which shows up $|\mathbb{F}_q^\times|$ times.\\

So we have
\[
\sum_{a, c \in \mathbb{F}_q^{\times}}\begin{pmatrix}
    -a/c & 0 \\
    -(a+c)/ac & -c/a 
    \end{pmatrix} = 
\sum_{\substack{a \in \mathbb{F}_q^\times \setminus \{1\}\\ t\in \mathbb{F}_q^{\times}}} \begin{pmatrix}
a & 0 \\
t & a^{-1}
\end{pmatrix} + (q-1) I.\]

This simplifies to
\begin{align*}
    \sum_{a, c \in \mathbb{F}_q^{\times}}\begin{pmatrix}
    -a/c & 0 \\
    -(a+c)/ac & -c/a 
    \end{pmatrix} &= \sum_{a \in \mathbb{F}_q^\times} \begin{pmatrix}
    a & 0 \\
    0 & a^{-1}
    \end{pmatrix} (qe_{U_{\lambda}^{op}} - I)- qe_{U_{\lambda}^{op}}+q I.
\end{align*}
Plugging this back into our original equation we get

\begin{multline*}
    \R^2 =  q^4e_{U_{\lambda}}e_{U_{\lambda}^{op}}e_{U_{\lambda}} -  (q-1) q^2e_{U_{\lambda}} +  (q-1) q^2e_{U_{\lambda}}e_{U_{\lambda}^{op}}e_{U_{\lambda}}-\\
    - qe_{U_{\lambda}} \left( \sum_{a \in \mathbb{F}_q^\times} \begin{pmatrix}
    a & 0 \\
    0 & a^{-1}
    \end{pmatrix} (qe_{U_{\lambda}^{op}} - I)- qe_{U_{\lambda}^{op}}+q I\right) e_{U_{\lambda}}+\\
    + qe_{U_{\lambda}} \left(  \sum_{a \in \mathbb{F}_q^\times} \begin{pmatrix}
    a & 0 \\
    0 & a^{-1}
    \end{pmatrix} (qe_{U_{\lambda}} - I)- qe_{U_{\lambda}}+q I \right) e_{U_{\lambda}}.
\end{multline*}

We can distribute out and cancel a few terms:

\begin{multline*}
    \R^2 =  q^4e_{U_{\lambda}}e_{U_{\lambda}^{op}}e_{U_{\lambda}} - q^3 e_{U_{\lambda}} + q^3e_{U_{\lambda}}e_{U_{\lambda}^{op}}e_{U_{\lambda}} -\\- qe_{U_{\lambda}} \left( \sum_{a \in \mathbb{F}_q^\times} \begin{pmatrix}
    a & 0 \\
    0 & a^{-1}
    \end{pmatrix} (qe_{U_{\lambda}^{op}} - I)\right) e_{U_{\lambda}}+
     qe_{U_{\lambda}} \left(  \sum_{a \in \mathbb{F}_q^\times} \begin{pmatrix}
    a & 0 \\
    0 & a^{-1}
    \end{pmatrix} (qe_{U_{\lambda}} - I) \right) e_{U_{\lambda}}.
\end{multline*}

Now, we can factor out the Levi part from the equation

\begin{equation*}
     \R^2=  q^4e_{U_{\lambda}}e_{U_{\lambda}^{op}}e_{U_{\lambda}} - q^3 e_{U_{\lambda}} + q^3e_{U_{\lambda}}e_{U_{\lambda}^{op}}e_{U_{\lambda}} + q^2\left(  \sum_{a \in \mathbb{F}_q^\times} \begin{pmatrix}
    a & 0 \\
    0 & a^{-1}
    \end{pmatrix} \right) ( e_{U_{\lambda}} - e_{U_{\lambda}}e_{U_{\lambda}^{op}}e_{U_{\lambda}}).
\end{equation*}

Now, we can substitute back in $\R = q^2e_{U_{\lambda}}e_{U_{\lambda}^{op}}e_{U_{\lambda}}$ to get the next equality:
\begin{equation*}
    \R^2 =  q^2 \R - q^3 e_{U_{\lambda}} + q \R + \left(  \sum_{a \in \mathbb{F}_q^\times} \begin{pmatrix}
    a & 0 \\
    0 & a^{-1}
    \end{pmatrix} \right) ( q^2 e_{U_{\lambda}} - \R).
\end{equation*}

We move some terms to the other side
\begin{equation*}
    \R^2 -  q^2 \R + q^3 e_{U_{\lambda}} - q \R  = -\left(  \sum_{a \in \mathbb{F}_q^\times} \begin{pmatrix}
    a & 0 \\
    0 & a^{-1}
    \end{pmatrix} \right) ( \R - q^2 e_{U_{\lambda}})
\end{equation*}
and square both sides
\begin{multline*}
    \left(\R^2 -  q^2 \R + q^3 e_{U_{\lambda}} - q \R\right)^2 =\\
    \left(  \sum_{a \in \mathbb{F}_q^\times} \begin{pmatrix}
    a & 0 \\
    0 & a^{-1}
    \end{pmatrix} \right) (q^2 e_{U_{\lambda}} - \R) \left(  \sum_{b \in \mathbb{F}_q^\times} \begin{pmatrix}
    b & 0 \\
    0 & b^{-1}
    \end{pmatrix} \right) ( q^2 e_{U_{\lambda}} - \R).
\end{multline*}

However, since $\R$ and $e_{U_{\lambda}}$ commute with diagonal matrices, we can actually write it as follows:
\begin{equation*}
    \left(\R^2 -  q^2 \R + q^3 e_{U_{\lambda}} - q \R\right)^2 =\\
    \left(  \sum_{a \in \mathbb{F}_q^\times} \begin{pmatrix}
    a & 0 \\
    0 & a^{-1}
    \end{pmatrix} \right)^2 ( q^2 e_{U_{\lambda}} - \R)^2.
\end{equation*}

Therefore,
\begin{equation*}
    \left(\R^2 -  q^2 \R + q^3 e_{U_{\lambda}} - q \R\right)^2 =
    (q-1) \left(  \sum_{a \in \mathbb{F}_q^\times} \begin{pmatrix}
    a & 0 \\
    0 & a^{-1}
    \end{pmatrix} \right) ( q^2 e_{U_{\lambda}} - \R)^2.
\end{equation*}

Note that the right side has a multiple that looks very similar to the equation we had before we squared it, so we can substitute that back in, allowing us to get rid of the sum
\begin{equation*}
    \left(\R^2 -  q^2 \R + q^3 e_{U_{\lambda}} - q \R\right)^2 =
    (1-q) \left(\R^2 -  q^2 \R + q^3 e_{U_{\lambda}} - q \R\right) ( \R - q^2e_{U_{\lambda}} ),
\end{equation*}
\begin{equation*}
    \left(\R^2 -  q^2 \R + q^3 e_{U_{\lambda}} - q \R\right)
    \left(\R^2 -  q^2 \R + q^3 e_{U_{\lambda}} - q \R-(1-q)(\R - q^2e_{U_{\lambda}})\right) = 0.
\end{equation*}

Finally, since we are studying Radon transform on $U_\lambda$-invariant functions, we may substitute $1$ for $e_{U_{\lambda}}$, and we can simplify and factorize this equation.

\begin{equation*}
    (\R - q^2)^2(\R - q)(\R -1) = 0.
\end{equation*}
 
From this we get that all the possible eigenvalues of $\R$ are $1$, $q$, and $q^2$. An explicit computation (e.g. for prime $q$) shows that every eigenvalue actually appears in the spectrum.
\end{proof}

\begin{prop}
    For the action of the Radon transform on $\CC[\GL{2}/U_{\lambda}]$, the eigenvalue $q^2$ occurs $q-1$ times, $q$ occurs $(q-1)(q^2-q-2)$ times, and $1$ occurs $(q-1)q$ times
\end{prop}

\begin{proof}
    Including multiplicity, there are a total of $|\GL{2} / U_\lambda| = (q^2-1)(q^2-q)/q = (q^2-1)(q-1)$ eigenvalues. The sum of eigenvalues is $|\GL{2}| = (q^2-1)(q^2-q)$ and the sum of squares of the eigenvalues is $(q^2-1)(q^2-q)(2q-1)$ due to the Proposition \ref{trace_and_square}. Let the multiplicity of the eigenvalue $q^2$ be $a$, $q$ be $b$, and $1$ be $c$. Then, we have the equations
    \begin{align*}
        a+b+c &= (q^2-1)(q-1)\\
        q^2a+qb+c &= (q^2-1)(q^2-q)\\
        q^4a+q^2b+c &= (q^2-1)(q^2-q)(2q-1)
    \end{align*}

    This can be solved to obtain our desired result. (note that $a, b$, and $c$ must be divisible by $q-1$ by \thmref{blocks} which makes this system of equations slightly easier)
\end{proof}

\section{Hecke Algebras}

Now, we look into Hecke Algebras, defined in \ref{finite_group_Hecke_algebra}. Certain Hecke algebras are well studied, and we take an overview of the relations here.\\

The case of double coset algebra $\mathcal{H}(\GL{n},B_n)$ for $\GL{n}$ with respect to the Borel group $B_n$ given by upper triangular matrices leads us specifically to Iwahori-Hecke algebra, which is an algebra defined by generators $g_1, g_2, \dots g_{n-1}$ (e.g. see \cite{DB}) subject to the relations
\begin{align*}
g_i g_{i+1} g_i &= g_{i+1}g_i g_{i+1},\\
g_i g_j &= g_j g_i \quad (|i-j| > 1),\\
g_i^2 &= (q-1)g_i+q.
\end{align*}

The first two are known as the braid relations, which are closely related to the symmetric group. Each $g_i$ corresponds to an elementary transposition $(i,i+1)$ in the Bruhat decomposition as in \thmref{bruhat}
\begin{equation}\label{Bruhat_decomposition_B}
    \GL{n}=\bigsqcup_{w\in S_n} B_n w B_n.
\end{equation}
Remarkably, this set of braid relations with specialization to $q=1$ gives us a realization of the symmetric group algebra $\CC[S_n]$.\\

It is known that the Iwahori-Hecke algebra describes convolution of Bruhat cells $B_n w B_n$ in the Bruhat decomposition, as shown in Theorem 11 in \cite{DB}. 

\subsection{Yokonuma-Hecke Algebras}

Now, we consider a generalization of the Iwahori-Hecke algebra known as the Yokonuma-Hecke Algebra. It is the Hecke algebra $\mathcal{H}(\GL{n},U_{\lambda_n})$ associated with the unipotent radical $U_{\lambda_n}$ of the Borel subgroup $B_n\subset \GL{n}$ or the subgroup of upper triangular matrices with $1$'s on the diagonal. Recall from Theorem \ref{unipotent_Bruhat} that here we have an analogue of the Bruhat decomposition given by
\begin{equation}\label{Bruhat_decomposition_U}
    \GL{n}=\bigsqcup_{\nu\in S_n \ltimes T_n} U_{\lambda_n} \nu U_{\lambda_n}.    
\end{equation}



We have the following Theorem describing generators and relations for the Yokonuma-Hecke algebra $\mathcal{H}(\GL{n},U_{\lambda_n})$. Note that we use a different normalization of the generators $g_i$.
\begin{thm}[\cite{JJ,Yok}]\label{Yok_generators}
    Let $q$ be a power of an odd prime, $\alpha$ is a primitive generator of $\mathbb{F}_q^{\times}$ and let $s_i$ be the permutation matrix corresponding to the elementary transposition $(i,i+1)$. The Yokonuma-Hecke algebra $\mathcal{H}(\GL{n},U_{\lambda_n})$ is generated by elements $g_i:=e_{U_{\lambda_n}} s_ie_{U_{\lambda_n}},\ 1\le i<n$ and $t_i:=e_{U_{\lambda_n}} {\rm diag}(1^{(i-1)},\alpha, 1^{(n-i)})e_{U_{\lambda_n}}$ with the following list of relations
    \begin{align}
        g_i g_{i+1} g_i &= g_{i+1}g_i g_{i+1},\\\label{braid_1}
        g_i g_j &= g_j g_i \quad (|i-j| > 1),\\\label{braid_2}
        t_i t_j &= t_j t_i,\\\label{torus_1}
        t_j g_i &= g_i t_{s_i(j)},\\\label{torus-permutation}
        t_j^{q-1} &= 1,\\\label{torus_2}
        g_i^2 = q^{-1}(1 +& \sum_{k=1}^{q-1}t_i^kt_{i+1}^{-k+(q-1)/2} g_i).
    \end{align}
    If $q$ is a power of $2$, then the last relation is replaced by
    \begin{equation}
        g_i^2 = q^{-1}(1 + \sum_{k=1}^{q-1}t_i^kt_{i+1}^{-k} g_i).
    \end{equation}
\end{thm}

For reader's convenience we will provide a proof of this Theorem below.

\begin{proof}
    By using the Lemma \ref{word_multiplication} below as well as the analogue of the Bruhat decomposition \eqref{Bruhat_decomposition_U} we can see that $g_i,t_i$ generate the Yokonuma-Hecke algebra and elements of the form $e_{U_{\lambda_n}} t w e_{U_{\lambda_n}}, t\in T_n,w\in S_n$ form a basis of the algebra. The braid relations \eqref{braid_1}, \eqref{braid_2} also follow from the Lemma \ref{word_multiplication}. Since $U_{\lambda_n}$ is a normal subgroup of $B_n$ the elements $e_Ute_U,t\in T_n$ satisfy the same relations \eqref{torus_1}, \eqref{torus_2} as $t\in T_n\subset \GL{n}$ and they commute with the permutation elements $g_i$ up to a permutation action \eqref{torus-permutation} on the labels. In order for us to describe the multiplication table of $\mathcal{H}(\GL{n},U_{\lambda_n})$ we only have to show the last relation due to the existence of a reduced decomposition of $w\in S_n$ in terms of $s_i$ (recall the Theorem \ref{reduced_decomposition}). \\

Note that the application of $U_{\lambda_n}$ on the left corresponds to the "upward" row addition. Thus, if $s_i$ is an elementary permutation matrix

\[\begin{pmatrix}
1 &  &  &  &  &  \\
 & \ddots &  &  &  &  \\
 &  & 0 & 1 &  &  \\
 &  & 1 & 0 &  &  \\
 &  &  &  & \ddots &  \\
 &  &  &  &  & 1 
\end{pmatrix},\]

then $U_{\lambda_n} s_i$ corresponds to the set

\[\begin{pmatrix}
1 &  & * & * &  & * \\
 & \ddots &  &  &  &  \\
 &  & * & 1 &  & * \\
 &  & 1 & 0 &  & * \\
 &  &  &  & \ddots &  \\
 &  &  &  &  & 1 
\end{pmatrix}.\]

Therefore, $s_i U_{\lambda_n} s_i$ corresponds to the set

\[\begin{pmatrix}
1 &  & * & * &  & * \\
 & \ddots &  &  &  &  \\
 &  & 1 & 0 &  & * \\
 &  & a & 1 &  & * \\
 &  &  &  & \ddots &  \\
 &  &  &  &  & 1 
\end{pmatrix}\]
for any $a\in \mathbb{F}_q$.\\

Now, for any such matrix we need to figure out its corresponding double $U_{\lambda_n}$ coset. Note that we can factor out elements of $U_{\lambda_n}$ by adding rows upwards and columns to the right. If $a = 0$, then this is clearly an element of $U_{\lambda_n}$ and corresponds to the double coset $U_{\lambda_n} eU_{\lambda_n}$. Otherwise, we can add the $i+1$th row upwards with a multiple of $-a^{-1}$ to get

\[\begin{pmatrix}
1 &  & * & * &  & * \\
 & \ddots &  &  &  &  \\
 &  & 0 & -a^{-1} &  & * \\
 &  & a & 1 &  & * \\
 &  &  &  & \ddots &  \\
 &  &  &  &  & 1 
\end{pmatrix}.\]

From here, it is a simple process to add all the rows rightward in order to obtain the representative 

\[\begin{pmatrix}
1 &  &  &  &  &  \\
 & \ddots &  &  &  &  \\
 &  & 0 & -a^{-1} &  &  \\
 &  & a & 0 &  &  \\
 &  &  &  & \ddots &  \\
 &  &  &  &  & 1 
\end{pmatrix}.\]

This, in terms of generators, corresponds to the element $t_i^k t_{i+1}^{-k+(q-1)/2} g_i$ with a unique $k$ such that $\alpha^{-k+(q-1)/2}=a$. Also, note that in the set $s_i U_{\lambda_n} s_i$ the value of $a$ runs over $\mathbb{F}_q$ with equal frequency. Thus, we can conclude that the element $g_i^2$ is proportional to $e_{U_{\lambda_n}} + \sum_{k=1}^{q-1}t_i^kt_{i+1}^{-k+(q-1)/2}g_i$.\\ 

Now we note that $\epsilon(g_i)=\epsilon(t_i^k t_{i+1}^{-k+(q-1)/2}g_i)=\epsilon(e_{U_{\lambda_n}})=1$ where $\epsilon$ is the counit of the group algebra $\CC[\GL{n}]$. On the other hand, due to the Proposition \ref{counit} we have $\epsilon(g_i^2)=\epsilon(g_i)=1$. Therefore, we should get
\[
g_i^2=q^{-1}(1 + \sum_{k=1}^{q-1}t_i^kt_{i+1}^{-k+(q-1)/2} g_i).
\]
For the case when $q$ is a power of $2$, observe that $-a^{-1}=a^{-1}$ in $\mathbb{F}_q$.
\end{proof}



Now, we consider relations between unipotent Bruhat cells. To do this, we modify Lemma 6 from \cite{DB}.

\begin{lem}\label{word_multiplication}
    Let $w\in S_n$, then $U_{\lambda_n} w U_{\lambda_n} \cdot U_{\lambda_n} s_iU_{\lambda_n} = U_{\lambda_n} ws_i U_{\lambda_n}$ when $l(ws_i)=l(w)+1$.
\end{lem}


\begin{proof}
    Clearly $U_{\lambda_n} ws_iU_{\lambda_n} \subseteq U_{\lambda_n} w U_{\lambda_n} \cdot U_{\lambda_n} s_iU_{\lambda_n}$. We will show that $wU_{\lambda_n}s_i\subseteq U_{\lambda_n} w s_i U_{\lambda_n}$, as multiplying both sides by $U_{\lambda_n}$ gives us our desired inclusion.\\

    Let $\alpha_i$ be the simple root corresponding to $\epsilon_i-\epsilon_{i+1}$. For any root $\alpha\in \Phi,\ \alpha=\epsilon_i-\epsilon_j$ we will denote by
    \begin{equation}\label{1-dim_subgroup}
        x_{\alpha}(\lambda)=1+\lambda E_{ij}, \quad \lambda\in \mathbb{F}_q,
    \end{equation}
    the elementary matrix with the entry $\lambda$ at $(i,j)$.\\
    
    Let $u \in U_{\lambda_n}$, then note that we can write $u$ as $x_{\alpha_i}(\lambda) v$ for some $v \in U_{\alpha_i}:=U_{(1^{(i-1)},2,1^{(n-i-1)})}$ and $\lambda\in \mathbb{F}_q$.\\

    Then, $wus_i = w x_{\alpha_i}(\lambda) w^{-1} \cdot w s_i \cdot s_i^{-1} v s_i$. We have that $s_i^{-1} v s_i \in U_{\alpha_i}\subset U_{\lambda_n}$ because $s_i$ is an element of $M_{(1^{(i-1)},2,1^{(n-i-1)})}$. In addition, we have $w x_\alpha(\lambda) w^{-1} \in x_{w(\alpha)}(\mathbb{F}_q) \subset U_{\lambda_n}$ as well due to the Theorem \ref{reduced_decomposition}, so $wus_i \in U_{\lambda_n} w s_i U_{\lambda_n}$ and we are done.
\end{proof}



\subsection{Gelfand-Graev Hecke Algebras}

Gelfand-Graev Hecke algebras are very similar to Yokonuma-Hecke algebras, with both falling under the class of unipotent Hecke algebras (see \cite{T}).\\


Let $U_{\lambda_n}$ be the unipotent subgroup of $\GL{n}$
. Also, let us fix a primitive additive character $\psi: (\mathbb{F}_q,+)\rightarrow \CC^*$. For example, when $q=p^n$ we may take
\[
\psi(x):=e^{\frac{2\pi {\rm i}}{p}\tr_{\mathbb{F}_q/\mathbb{F}_p}(x)} , \quad \tr_{\mathbb{F}_q/\mathbb{F}_p} (x):=x+x^p+x^{p^2}\dots+x^{p^{n-1}}\in \mathbb{F}_p.
\]
Recall that any other additive character of $\mathbb{F}_q$ can be obtained from it by $\psi(-)\mapsto \psi_{a}(-):=\psi(a\cdot -),a\in \mathbb{F}_q$. Indeed, note that since the trace map is not identically zero (the finite extension $\mathbb{F}_q/\mathbb{F}_p$ is separable), any such character is non-trivial for $a\neq 0$. We also have $\psi_a(x)\psi_b(x)=\psi_{a+b}(x)$ for $a,b,x\in \mathbb{F}_q$.\\

For a tuple $\mu$ of $n-1$ elements $\mu_{i}\in \mathbb{F}_q, 1\le i<n$ we define a nontrivial linear character of $U_{\lambda_n}$
\[\psi_{\mu}: U_{\lambda_n} \to \mathbb{C}^*, \quad x_{\alpha}(t)\mapsto \psi(\mu_{\alpha}t)\]
where $x_{\alpha}(t)$ is defined as in \eqref{1-dim_subgroup}, $\mu_{\alpha}$ is $0$ if $\alpha$ is not a simple root and $\mu_i$ if $\alpha$ is a simple root $\epsilon_i-\epsilon_{i+1}$.\\

Now, we define $e_{\psi_{\mu}}$ as

\[e_{\psi_{\mu}} = \frac{1}{|U_{\lambda_n}|}\sum_{u \in U_{\lambda_n}}\overline{\psi_{\mu}(u)}u\]

This is an projector as defined in \ref{e-projector}, and $e_{\psi_s}$ corresponding to different characters are orthogonal.

\begin{defn}
    The Gelfand-Graev Hecke algebra $\mathcal{H}(\GL{n}, U_{\lambda_n}, \psi_{\mu})$ is a subalgebra $e_{\psi_{\mu}} \mathbb{C}[\GL{n}] e_{\psi_{\mu}}\subset \CC[\GL{n}]$.    
\end{defn}

Note that in the case of the trivial linear character, the corresponding Gelfand-Graev Hecke algebra is precisely the Yokonuma-Hecke algebra.\\

We recall the following fact about Gelfand-Graev Hecke algebras (see \cite[Proposition 11.30]{CR}).\\

\begin{prop}\label{GGH_basis}
    Let 
    \begin{align*}
        N_{\mu}&:=\{\nu \in S_n \ltimes T_n\ |\ e_{\psi_{\mu}}\nu e_{\psi_{\mu}}\neq 0\}=\\
        &=\{\nu\in S_n \ltimes T_n\ | \ u,\nu u\nu^{-1}\in U \Rightarrow \psi_{\mu}(u)=\psi_{\mu}(\nu u\nu^{-1})\},
    \end{align*}
    then the set $\{e_{\psi_{\mu}} \nu e_{\psi_{\mu}}\ |\ \nu \in N_{\mu}\}$ forms a basis for $\mathcal{H}(\GL{n},U_{\lambda_n},\psi_{\mu})$.
\end{prop}

\section{The Case of $\GL{3}$}

In this section we find the eigenvalues of the Radon transform in the case of $\lambda=(2,1)$. The case of $\lambda=(1,2)$ can be treated similarly. In order for us to reach our goal we will have to work with the Yokonuma-Hecke algebra for $\GL{3}$ (which covers the only nontrivial case left corresponding to the partition $\lambda=(1,1,1)$) and a certain family of Gelfand-Graev Hecke algebras.\\


\subsection{Hecke Algebra $\mathcal{H}(\GL{3},U_{(2,1)})$}

In this subsection we will work with the Radon transform for the Hecke algebra $\mathcal{H}(\GL{3},U_{(2,1)})$. Consider the set of orthogonal idempotents (recall the Definition \ref{idempotents})
\begin{equation*}
    e_{\psi_s}:= q^{-1}\sum_{a\in \mathbb{F}_q} \overline{\psi(sa)} e_{U_{(2,1)}} x_{12}(a)e_{U_{(2,1)}}=q^{-1}e_{U_{(2,1)}}\sum_{a\in \mathbb{F}_q} \overline{\psi(sa)} x_{12}(a)=e_{\psi_{(s,0)}} \in \CC[\GL{n}].
\end{equation*}
enumerated by elements $s\in \mathbb{F}_q$. Note that the subalgebra $e_{\psi_0}\mathcal{H}(\GL{3},U_{(2,1)})e_{\psi_0}$ is isomorphic to the Yokonuma-Hecke algebra $\mathcal{H}(\GL{3},U_{\lambda_3})$ and each subalgebra\linebreak $e_{\psi_s}\mathcal{H}(\GL{3},U_{(2,1)})e_{\psi_s}$ for $s\neq 0$ is isomorphic to the Gelfand-Graev Hecke algebra $\mathcal{H}(\GL{3},U_{\lambda_3},\psi_{(s,0)})$.\\

We also have the decomposition of the identity as in the following Proposition









\begin{prop}\label{sum_of_idempotents}
    The following equality holds true
    \[
    \sum_{s\in \mathbb{F}_q}e_{\psi_s} = e_{U_{(2,1)}}\in \CC[\GL{3}]
    \]
\end{prop}

\begin{proof}
    This is due to the fact that
    \[
    \sum_{s\in \mathbb{F}_q}e_{\psi_s}=e_{U_{(2,1)}}\sum_{s\in \mathbb{F}_q}q^{-1}\sum_{a\in \mathbb{F}_q}\overline{\psi(sa)}x_{12}(a)=e_{U_{(2,1)}}\cdot 1=e_{U_{(2,1)}}.
    \]
    Here, the second equality follows from the Proposition \ref{e-projector} for an abelian subgroup corresponding to $x_{12}$.
\end{proof}

We wish to describe other properties of $e_{\psi_s}$.

\begin{prop}\label{Epsi*N}
    \begin{equation*}
        e_{\psi_s}
        \begin{pmatrix}
            a & &\\
            & b &\\
            & & c
        \end{pmatrix} = 
        \begin{pmatrix}
            a & &\\
            & b &\\
            & & c
        \end{pmatrix}
        e_{\psi_{sab^{-1}}}, \quad a,b,c\in \mathbb{F}_q^{\times}
    \end{equation*}
\end{prop}

\begin{proof}
    This can be shown by simple expansion, and substituting $x' = ba^{-1}x$

    \begin{align*}
        e_{\psi_s} \begin{pmatrix}
            a & &\\
            & b &\\
            & & c
        \end{pmatrix} &= 
        \sum_{x, y, z \in \mathbb{F}_q}
        \begin{pmatrix}
            1 & x & y\\
            & 1 & z\\
            & & 1
        \end{pmatrix}
        \begin{pmatrix}
            a & & \\
            & b & \\
            & & c
        \end{pmatrix}
        \overline{\psi_s(x)}\\
        &= \sum_{x, y, z \in \mathbb{F}_q}\begin{pmatrix}
            a & bx & cy\\
            & b & cz\\
            & & c
        \end{pmatrix} \overline{\psi_s(x)}\\
        &= \begin{pmatrix}
            a & & \\
            & b &\\
            & & c
        \end{pmatrix} \sum_{x, y, z \in \mathbb{F}_q}
        \begin{pmatrix}
            1 & ba^{-1}x & ca^{-1}y\\
            & 1 & cb^{-1}z\\
            & & 1
        \end{pmatrix} \overline{\psi_s(x)}\\
        &= \begin{pmatrix}
            a & & \\
            & b &\\
            & & c
        \end{pmatrix} \sum_{x', y, z \in \mathbb{F}_q}
        \begin{pmatrix}
            1 & x' & ca^{-1}y\\
            & 1 & cb^{-1}z\\
            & & 1
        \end{pmatrix} \overline{\psi_s(xab^{-1})}\\
        &= \begin{pmatrix}
            a & & \\
            & b &\\
            & & c
        \end{pmatrix} \sum_{x', y, z \in \mathbb{F}_q}
        \begin{pmatrix}
            1 & x' & ca^{-1}y\\
            & 1 & cb^{-1}z\\
            & & 1
        \end{pmatrix} \overline{\psi_{sab^{-1}}(x)}\\
        &= \begin{pmatrix}
            a & & \\
            & b &\\
            & & c
        \end{pmatrix} e_{\psi_{sab^{-1}}}.
    \end{align*}
\end{proof}

We will also need the following result on element vanishing.

\begin{prop}\label{g2is0}
    The element 
    \[
    e_{\psi_s} t \begin{pmatrix}
        1 & & \\
        & & 1\\
        & 1 &
    \end{pmatrix} e_{\psi_r}
    \]
    vanishes for any $s$, $r \in \mathbb{F}_q$ other than $s = r = 0$, and all $t \in T$.
\end{prop}

\begin{proof}
    Note that $t$ does not matter here, as by the Proposition \ref{Epsi*N}, we can commute $t$ through $e_{\psi_s}$ (also note that the character changes, but we need to show this for any character anyways, so this is a valid operation). We show the statement by expansion
        \begin{align*}
        e_{\psi_s}
        \begin{pmatrix}
            1 & &\\
            &  & 1\\
            & 1 & 
        \end{pmatrix} e_{\psi_r} &= 
        \sum_{x, y, z \in \mathbb{F}_q}
        \begin{pmatrix}
            1 & x & y\\
            & 1 & z\\
            & & 1
        \end{pmatrix} \overline{\psi_s(x)}
        \begin{pmatrix}
            1 & &\\
            & & 1\\
            & 1 &
        \end{pmatrix}
        \sum_{x', y', z' \in \mathbb{F}_q}
        \begin{pmatrix}
            1 & x' & y'\\
            & 1 & z'\\
            & & 1
        \end{pmatrix} \overline{\psi_r(x')}\\
        &=
        \sum_{x, y, z, x', y', z' \in \mathbb{F}_q}
        \begin{pmatrix}
            1 & y + x' & x+ yz'+y'\\
            0 & z & zz'+1\\
            0 & 1 & z'
        \end{pmatrix}
        \overline{\psi_s(x) \psi_r(x')}.
    \end{align*}

    Now, assign $y'' = y + x'$ and $y''' = y'+x-x'z'$ to replace $y$ and $y'$. Note that as $y$ goes over all of $\mathbb{F}_q$, $y''$ does so as well for any $x'$, and similarly for $y'''$ and $y'$. Thus, we can rewrite the sum as follows

    \begin{equation*}
        \sum_{x, z, x', z', y'', y''' \in \mathbb{F}_q}
        \begin{pmatrix}
            1 & y'' & y''z'+y'''\\
            0 & z & zz'+1\\
            0 & 1 & z'
        \end{pmatrix}
        \overline{\psi_s(x) \psi_r(x')}.
    \end{equation*}

    Now, we can separate $x$ and $x'$ to get

    \begin{equation*}
        \sum_{z, z', y'', y''' \in \mathbb{F}_q} \begin{pmatrix}
            1 & y'' & y''z'+y'''\\
            0 & z & zz'+1\\
            0 & 1 & z'
        \end{pmatrix}
        \sum_{x \in \mathbb{F}_q}\overline{\psi_s(x)} \sum_{x' \in \mathbb{F}_q}\overline{\psi_r(x')}.
    \end{equation*}

    The sum $\sum_{x \in \mathbb{F}_q} \overline{\psi_s(x)} $ is $0$ if $s \ne 0$, and similarly for $r$, so we can say that this sum vanishes whenever $s$ and $r$ are not both $0$.
\end{proof}

Finally, we turn to the expansion of the Radon transform in terms of the Hecke algebra $\mathcal{H}(\GL{3},U_{(2,1)})$. 

\begin{prop}\label{Radon_transform_decomposition}
    The Radon transform $\R$ in $\mathcal{H}(\GL{3}, U_\lambda)$ for $\lambda = (2,1)$ can be expressed as the sum of elements from unipotent Hecke subalgebras corresponding to the characters $\psi_s,s\in\mathbb{F}_q$. Specifically, $\R$ can be expressed as follows

    \begin{equation*}
        q^2\left(\sum_{b \in \mathbb{F}_q^\times}e_{\psi_0} \begin{pmatrix}
            1 & &\\
            & & -b^{-1}\\
            & b &
        \end{pmatrix}e_{\psi_0} + 
        \sum_{s \in \mathbb{F}_q}\left( e_{\psi_s} 1 e_{\psi_s} + q\sum_{a \in \mathbb{F}_q^\times}e_{\psi_s}\begin{pmatrix}
            & & -a^{-1}\\
            & 1 &\\
            a & &
        \end{pmatrix} e_{\psi_s}\right)\right)
    \end{equation*}
\end{prop}


\begin{proof} We recall the formula for the Radon transform
\[
\R= |U_{\lambda}|^2e_{U_{\lambda}}e_{U_{\lambda}^{op}}e_{U_{\lambda}}=q^2e_{U_{\lambda}}\sum_{a,b\in\mathbb{F}_q}\begin{pmatrix}
    1 & 0 & 0\\
    0 & 1 & 0\\
    a & b & 1
\end{pmatrix}
e_{U_{\lambda}}
\]





According to the Proposition \ref{sum_of_idempotents} we may decompose it further
    \[\R =q^2\sum_{s,t\in \mathbb{F}_q} \left( e_{\psi_s} 1 e_{\psi_t} + \sum_{b \in \mathbb{F}_q^\times} e_{\psi_s} 
    \begin{pmatrix}
        1 & & \\
        & 1 & \\
        & b & 1
    \end{pmatrix} e_{\psi_t}+ \sum_{\substack{a \in \mathbb{F}_q^\times\\b \in \mathbb{F}_q}} e_{\psi_s}
    \begin{pmatrix}
        1 & & \\
         & 1 & \\
        a & b & 1
    \end{pmatrix} e_{\psi_t}\right)\]
    
    We can work with the first term. We already know that $e_{\psi_s} e_{\psi_t} = 0$ if $s \ne t$ due to the fact that $e_{\psi_s}$ and $e_{\psi_t}$ are orthogonal idempotents.\\
    
    We can apply the Bruhat decomposition to simplify the expression
    \[
    \begin{pmatrix}
        1 & & \\
        & 1 &\\
        & b & 1
    \end{pmatrix} = 
    \begin{pmatrix}
        1 & & \\
        & 1 & b^{-1}\\
        & & 1
    \end{pmatrix}
    \begin{pmatrix}
        1 & &\\
        & & -b^{-1}\\
        & b &
    \end{pmatrix}
    \begin{pmatrix}
        1 & &\\
        & 1 & b^{-1}\\
        & & 1
    \end{pmatrix}
    \]
    
    Note that, by the Proposition \ref{Epsi*U}

    \begin{equation*}
        \begin{pmatrix}
            1 & & \\
            & 1 & b^{-1}\\
            & & 1
        \end{pmatrix}
        e_{\psi_s} = e_{\psi_s} = e_{\psi_s}
        \begin{pmatrix}
            1 & & \\
            & 1 & b^{-1} \\
            & & 1
        \end{pmatrix}.
    \end{equation*}

    This allows us to change the second term to the following

    \begin{equation*}
        e_{\psi_s}
        \begin{pmatrix}
            1 & &\\
            &  & -b^{-1}\\
            & b & 
        \end{pmatrix} e_{\psi_t}.
    \end{equation*}

    It vanishes whenever $s$ and $t$ are not both zero by the Proposition \ref{g2is0}.\\

    The third term is a bit more complicated. We begin by writing the matrix's Bruhat decomposition

    \begin{equation*}
        \begin{pmatrix}
            1 & & \\
            & 1 &\\
            a & b & 1
        \end{pmatrix} = \begin{pmatrix}
            1 & -a^{-1}b & a^{-1} \\
            & 1 & \\
            & & 1
        \end{pmatrix} \begin{pmatrix}
            & & -a^{-1}\\
            & 1 &\\
            a & &
        \end{pmatrix} \begin{pmatrix}
            1 & a^{-1}b & a^{-1}\\
            & 1 & \\
            & & 1
        \end{pmatrix}.
    \end{equation*}

    Note that

    \begin{equation*}
        \begin{pmatrix}
            1 & -a^{-1}b & a^{-1} \\
            & 1 & \\
            & & 1
        \end{pmatrix} = \begin{pmatrix}
            1 & & a^{-1}\\
            & 1 & \\
            & & 1
        \end{pmatrix} \begin{pmatrix}
            1 & -a^{-1}b &\\
            & 1 & \\
            & & 1
        \end{pmatrix}
    \end{equation*}

    and 

    \begin{equation*}
        \begin{pmatrix}
            1 & a^{-1}b & a^{-1} \\
             & 1 & \\
            & & 1
        \end{pmatrix} = \begin{pmatrix}
            1 & a^{-1}b & \\
            & 1 & \\
            & & 1
        \end{pmatrix} \begin{pmatrix}
            1 & & a^{-1}\\
            & 1 & \\
            & & 1
        \end{pmatrix}.
    \end{equation*}

    The full term then turns into

    \begin{align*}
        \sum_{\substack{a \in \mathbb{F}_q^\times\\b \in \mathbb{F}_q}}e_{\psi_s} \begin{pmatrix}
            1 & -a^{-1}b & a^{-1} \\
            & 1 & \\
            & & 1
        \end{pmatrix}& \begin{pmatrix}
            & & -a^{-1}\\
            & 1 &\\
            a & &
        \end{pmatrix} \begin{pmatrix}
            1 & a^{-1}b & a^{-1}\\
            & 1 & \\
            & & 1
        \end{pmatrix} e_{\psi_t}\\
        = \sum_{\substack{a \in \mathbb{F}_q^\times\\b \in \mathbb{F}_q}}e_{\psi_s}\begin{pmatrix}
            1 & -a^{-1}b & \\
            & 1 & \\
            & & 1
        \end{pmatrix}&\begin{pmatrix}
            & & -a^{-1}\\
            & 1 &\\
            a & &
        \end{pmatrix} \begin{pmatrix}
            1 & a^{-1}b &\\
            & 1 & \\
            & & 1
        \end{pmatrix} e_{\psi_t}.
    \end{align*}

    According to the Proposition \ref{Epsi*U} we can rewrite the latter as
    \[
    \sum_{\substack{a \in \mathbb{F}_q^\times\\b \in \mathbb{F}_q}}e_{\psi_s}\overline{\psi_s(a^{-1}b)}\begin{pmatrix}
            & & -a^{-1}\\
            & 1 &\\
            a & &
        \end{pmatrix} \overline{\psi_t(-a^{-1}b)}e_{\psi_t}
    \]
    or as
    \[
    \sum_{b'\in \mathbb{F}_q}\overline{\psi_s(b')}\psi_t(b')\cdot\sum_{a \in \mathbb{F}_q^\times}e_{\psi_s}\begin{pmatrix}
            & & -a^{-1}\\
            & 1 &\\
            a & &
        \end{pmatrix} e_{\psi_t}.
    \]
    The factor $\sum_{b'\in \mathbb{F}_q}\overline{\psi_s(b')}\psi_t(b')$ is $0$ if $s\neq t$ and $q$ otherwise.
\end{proof}

\begin{cor}
    The set of eigenvalues of $\R$ is the union of the eigenvalues of $\R$ for the individual unipotent Hecke subalgebras.
\end{cor}

\subsection{Radon Transform for the Yokonuma-Hecke Algebra}
We can begin with the case of $\psi_0$, in which case we end up working with the Yokonuma-Hecke algebra.\\ 

From the Proposition \ref{Radon_transform_decomposition} we see that we need to work with the following element of $\mathcal{H}(\GL{3},U_3)$
\[
q^2\left( \sum_{b \in \mathbb{F}_q^\times}e_{\psi_0} \begin{pmatrix}
            1 & &\\
            & & -b^{-1}\\
            & b &
        \end{pmatrix}e_{\psi_0} + 
         e_{\psi_0} 1 e_{\psi_0} + q\sum_{a \in \mathbb{F}_q^\times}e_{\psi_0}\begin{pmatrix}
            & & -a^{-1}\\
            & 1 &\\
            a & &
        \end{pmatrix} e_{\psi_0}\right).
\]








Alternatively, by using notation of the Theorem \ref{Yok_generators}, it corresponds to

\begin{equation}\label{Radon_transform_Yok}
    q^2e_{\psi_0}\left( \sum_{k=1}^{q-1} t_2^k t_3^{-k+(q-1)/2} g_2 +1+ q\sum_{k=1}^{q-1} t_1^k t_3^{-k+(q-1)/2} g_1 g_2 g_1\right)e_{\psi_0}.
\end{equation}

\begin{rem}
    We know that the Radon transform for $\mathcal{H}(\GL{n},U_{(n-1,1)})$ commutes with the Radon transform for $\mathcal{H}(\GL{n-1},U_{(n-2,1)})$ when we view $\GL{n-1}$ as a subgroup of $\GL{n}$ embedded in the upper leftmost $(n-1)\times (n-1)$ square. Interestingly enough, the element from \eqref{Radon_transform_Yok} has terms with permutation types $e,(23)$ and $(13)$ from $S_3$, while the Radon transform for $\GL{2}$ has terms associated with $e$ and $(12)$ from $S_3$. Since both transforms commute, this situation has a certain similarity with the Jucys–Murphy elements in $\CC[S_3]$. It would be an interesting problem to find the formula for the Radon transform in the $\mathcal{H}(\GL{n},U_{(n-1,1)})$ case and check the permutation type of the terms involved in the decomposition.
\end{rem}












\begin{lem}\label{NCAlgebra_relations}
    For odd $q$ let us denote $\sum_{k=1}^{q-1} t_1^k t_2^{-k + (q-1)/2}$ as $l_{3}$, $\sum_{k=1}^{q-1} t_1^{k} t_3^{-k+(q-1)/2}$ as $l_2$, and \linebreak $\sum_{k=1}^{q-1} t_2^{k} t_3^{-k+(q-1)/2}$ as $l_1$. We can see that $\R$ can be expressed via $l_1,l_2,l_3,g_1$ and $g_2$. The elements $l_1,l_2,l_3,g_1$ and $g_2$ satisfy the following relations.
    \begin{align*}
        g_1g_2g_1 &= g_2g_1g_2,\\
        g_1^2 &= q^{-1}(1 + l_3g_1),\\
        g_2^2 &= q^{-1}(1 + l_1g_2),\\
        l_i g_1 &= g_1 l_{s_1(i)},\\
        l_i g_2 &= g_2 l_{s_2(i)},\\
        l_i^3 &= (q-1)^2l_i,\\
        l_i l_j &= l_j l_i,\\
        l_j l_i &= l_j l_k \quad (j \ne i, k).
    \end{align*}
    If $q$ is a power of $2$ we work with sums of the form $\sum_{k=1}^{q-1}t_i^kt_j^{-k}$ instead. In that case we replace the cubic relation for $l_i$ with $l_i^2=(q-1)l_i$.
\end{lem}

\begin{proof}
    Follows from a direct computation and relations from the Theorem \ref{Yok_generators}.
\end{proof}

We have the following Theorem.

\begin{thm}\label{eigenvalues_gl3_Yok}
    The eigenvalues of the Radon transform of $\lambda = (2,1)$ in the Yokonuma-Hecke algebra form a subset of $\{1, q, q^2, q^3,q^4\}$.
\end{thm}

\begin{proof}
    We can do this by showing that the minimal polynomial of the Radon transform is 
    \begin{equation}\label{Yok_Rd_relation}
        (\R - 1)(\R - q)(\R - q^2)(\R - q^3)(\R - q^4)    
    \end{equation}
     through a direct calculation.\\

    We can see that $\R$ can be expressed via $l_1,l_2,l_3,g_1$ and $g_2$ from the Lemma \ref{NCAlgebra_relations}:
    \[
    \R =q^2 (1+l_1g_2+ql_2g_1g_2g_1).
    \]
    We can see that the dimension of the algebra from Lemma \ref{NCAlgebra_relations} has a fixed dimension whenever $q$ is not a power of $2$ and a different fixed dimension whenever $q$ is a power of $2$. In fact, one can show that they are equal to $54$ and $30$ respectively.\\
    
    Therefore, these relations are sufficient to simplify $(\R - 1)(\R - q)(\R - q^2)(\R - q^3)(\R - q^4)$. By using the NCAlgebra mathematica package, we can demonstrate that $\R$ satisfies the relation given by the polynomial \eqref{Yok_Rd_relation} (see the Appendix for the Wolfram Mathematica code).\\ 
    
    A direct computation for the case $q = 3$ shows that all of the listed eigenvalues $1, 3, 9, 27$ and $81$ show up.
\end{proof}

As a result of the Theorem \ref{eigenvalues_gl3_Yok}, we get the Corollary below.\\

\begin{cor}\label{1_1_1}
    Observe that the Radon transform for $\GL{3},\lambda=(1,1,1)$ lies in the Hecke algebra $\mathcal{H}(\GL{3},U_{\lambda_3})\subset \mathcal{H}(\GL{3},U_{(2,1)})$, and thus it admits an action on $\CC[\GL{3}/U_{(2,1)}]$. The eigenvalues of such an action form a subset of $\{0,1,q,q^2,q^3,q^4,q^5,q^6\}$.
\end{cor}

\begin{proof}
    We combine the Theorems \ref{eigenvalues_box},\ref{eigenvalues_gl2} and \ref{eigenvalues_gl3_Yok} to get the list of possible eigenvalues. Note that the eigenvalue $0$ is possible because, for example $e_{\psi_s}\in \CC[\GL{3}/U_{(2,1)}]$ even if $s\neq 0$.
\end{proof}

\begin{rem}
    The list can be refined down to just $1,q,q^2,q^3,q^4,q^5$ and $q^6$ if we wish to work with $\R$ acting on $\CC[\GL{3}/U_{\lambda_3}]$ instead. We can apply the same Theorems to see this.
\end{rem}


    





\subsection{Radon Transform for the Gelfand-Graev Hecke Algebra}
The main result of this subsection is the following.\\

\begin{thm}\label{eigenvalues_GGH_gl3}
    The eigenvalues of $\R$ in the unipotent Hecke algebra corresponding to an idempotent $e_{\psi_s},s\neq 0$ form a subset of $\{1,q,q^2,q^3\}$. 
\end{thm}

    By the Proposition \ref{Radon_transform_decomposition}, in the unipotent Hecke algebra corresponding to an idempotent $e_{\psi_s},s\neq 0$,

    \begin{equation*}
        \R = q^2 e_{\psi_s} 1 e_{\psi_s} + q^3 \sum_{a \in \mathbb{F}_q^\times} e_{\psi_s} \begin{pmatrix}
            & & -a^{-1}\\
            & 1 &\\
            a & &
        \end{pmatrix} e_{\psi_s}
    \end{equation*}

    We claim that powers of $\R$ only require finitely many elements to express fully, and we will prove it by multiplying various elements by
    
    \[
        q^3 \sum_{a \in \mathbb{F}_q^\times} e_{\psi_s} \begin{pmatrix}
        & & -a^{-1}\\
        & 1 &\\
        a & &
    \end{pmatrix} e_{\psi_s}.
    \]
    
    Here are the calculations below. 


    \begin{lem}\label{B_3}
        \begin{align*}
            q^3\left(\sum_{a \in \mathbb{F}_q^\times} e_{\psi_s} \begin{pmatrix}
                & & -a^{-1}\\
                & 1 &\\
                a & &
            \end{pmatrix} e_{\psi_s}\right)^2 &= (q-1) e_{\psi_s} \begin{pmatrix}
                1 & & \\
                & 1 &\\
                & & 1
            \end{pmatrix} e_{\psi_s} +\\+ (q-1) \sum_{a, b \in \mathbb{F}_q^\times} e_{\psi_s} \begin{pmatrix}
                & a & \\
                -b & & \\
                & & a^{-1}b^{-1}
            \end{pmatrix} e_{\psi_s} &\overline{\psi_s(a + b^{-1})}+q(q-1)^2 \sum_{a \in \mathbb{F}_q^\times} e_{\psi_s} \begin{pmatrix}
                & & -a^{-1}\\
                & 1 &\\
                a & &
            \end{pmatrix} e_{\psi_s}-\\-& \sum_{c, d \in \mathbb{F}_q^\times} e_{\psi_s}
            \begin{pmatrix}
                & & -d^{-1}c^{-1}\\
                & d & \\
                c & &
            \end{pmatrix} e_{\psi_s}.
        \end{align*}
    \end{lem}

    \begin{proof} We start by expanding
    \begin{align*}
        q^3\left(\sum_{a \in \mathbb{F}_q^\times} e_{\psi_s} \begin{pmatrix}
        & & -a^{-1}\\
        & 1 &\\
        a & &
    \end{pmatrix} e_{\psi_s}\right)^2 &= q^3\sum_{a, b \in \mathbb{F}_q^\times} e_{\psi_s}\begin{pmatrix}
        & & -a^{-1}\\
        & 1 &\\
        a & &
    \end{pmatrix}e_{\psi_s}\begin{pmatrix}
        & & -b^{-1}\\
        & 1 &\\
        b & &
    \end{pmatrix}e_{\psi_s}\\
    = e_{\psi_s}\sum_{\substack{a, b \in \mathbb{F}_q^\times\\x, y, z \in \mathbb{F}_q}}&\begin{pmatrix}
        & & -a^{-1}\\
        & 1 &\\
        a & &
    \end{pmatrix} \begin{pmatrix}
        1 & x & y\\
        & 1 & z\\
        & & 1
    \end{pmatrix}\begin{pmatrix}
        & & -b^{-1}\\
        & 1 &\\
        b & &
    \end{pmatrix} \overline{\psi_s(x)} e_{\psi_s}\\
    = e_{\psi_s}\sum_{\substack{a, b \in \mathbb{F}_q^\times\\x, y, z \in \mathbb{F}_q}}&\begin{pmatrix}
        -ba^{-1} & &\\
        bz & 1 & \\
        aby & ax & -ab^{-1}
    \end{pmatrix} \overline{\psi_s(x)} e_{\psi_s}.
    \end{align*}

    Now, we split into two cases, based on whether $y = 0$ or not.\\ 

    \textbf{Case 1 ($y = 0$).} If $y = 0$, then according to the Proposition \ref{GGH_basis} the only possible permutations that can correspond to this expression which don't vanish are the identity permutation and the elementary transposition $(12)$. Thus, $x$ must also be equal to $0$.\\

    \textbf{Subcase 1 ($z = 0$).} If $z = 0$, then this becomes a very simple calculation:
    \begin{align*}
        e_{\psi_s} \sum_{a, b \in \mathbb{F}_q^\times} \begin{pmatrix}
            -ba^{-1} & & \\
            & 1 & \\
            & & -ab^{-1}
        \end{pmatrix} e_{\psi_s} &= (q-1) e_{\psi_s} \sum_{a' \in \mathbb{F}_q^\times} \begin{pmatrix}
            a' & & \\
            & 1 &\\
            & & a'^{-1}
        \end{pmatrix} e_{\psi_s}=\\
        &= (q-1)e_{\psi_s}\begin{pmatrix}
            1 & & \\
            & 1 & \\
            & & 1
        \end{pmatrix}e_{\psi_s}.
    \end{align*}

    Here, the last step follows from Propositions \ref{Epsi*N} and \ref{e-projector}.\\
    
    \textbf{Subcase 2 ($z \ne 0$).} We can write the Bruhat decomposition of the matrix to get
    \begin{align*}
        e_{\psi_s}\sum_{a, b, z \in \mathbb{F}_q^\times}\begin{pmatrix}
        -ba^{-1} & &\\
        bz & 1 & \\
         &  & -ab^{-1}
    \end{pmatrix} e_{\psi_s} &=\\= \sum_{a, b, z \in \mathbb{F}_q^\times} e_{\psi_s}\begin{pmatrix}
        1 & -a^{-1} z^{-1} & \\
        & 1 & \\
        & & 1
    \end{pmatrix}\begin{pmatrix}
        & a^{-1} z^{-1} &\\
        bz & & \\
        & & -ab^{-1}
    \end{pmatrix}&\begin{pmatrix}
        1 & b^{-1}z^{-1} &\\
        & 1 &\\
        & & 1
    \end{pmatrix} e_{\psi_s}=\\
    = \sum_{a, b, z \in \mathbb{F}_q^\times} e_{\psi_s} \begin{pmatrix}
        & a^{-1}z^{-1} &\\
        bz & & \\
        & & -ab^{-1}
    \end{pmatrix}&e_{\psi_s}\overline{\psi_s(a^{-1}z^{-1}-b^{-1}z^{-1})}.
    \end{align*}

    Now, if we let $a' = a^{-1}z^{-1}$ and $b' = -bz$ we get
    \begin{align*}
        \sum_{a, b, z \in \mathbb{F}_q^\times} e_{\psi_s} \begin{pmatrix}
        & a^{-1}z^{-1} &\\
        bz & & \\
        & & -ab^{-1}
    \end{pmatrix}e_{\psi_s}\overline{\psi_s(a^{-1}z^{-1}-b^{-1}z^{-1})} =\\= (q-1) \sum_{a', b' \in \mathbb{F}_q^\times} e_{\psi_s} \begin{pmatrix}
        & a' & \\
        -b' & & \\
        & & a'^{-1}b'^{-1}
    \end{pmatrix} e_{\psi_s} \overline{\psi(a' + b'^{-1})}.
    \end{align*}\\
    
    \textbf{Case 2 ($y \ne 0$).} Here we note that if $y = xz$, then the permutation matrix corresponding to the matrix in question is
    \[\begin{pmatrix}
        & 1 & \\
        & & 1 \\
        1 & &
    \end{pmatrix},
    \]
    so the respective element in the Gelfand-Graev Hecke algebra vanishes due to the Proposition \ref{GGH_basis}. Otherwise, we have the permutation 
    \[\begin{pmatrix}
        & & 1\\
        & 1 & \\
        1 & &
    \end{pmatrix}\]
    where the algebra element does not vanish. Thus, we can assume $y \ne xz$ and continue on, writing out the Bruhat decomposition of the matrix
    \begin{equation*}
        e_{\psi_s}\sum_{\substack{a, b \in \mathbb{F}_q^\times\\x, y, z \in \mathbb{F}_q\\y \ne xz}}\begin{pmatrix}
        -ba^{-1} & &\\
        bz & 1 & \\
        aby & ax & -ab^{-1}
    \end{pmatrix} \overline{\psi_s(x)} e_{\psi_s} =
    \end{equation*}
    \begin{align*}
    = e_{\psi_s}\sum_{\substack{a, b, y \in \mathbb{F}_q^\times\\x, z \in \mathbb{F}_q\\ y \ne xz}} \begin{pmatrix}
        1 & (1-y^{-1}xz)^{-1}a^{-1}y^{-1}x & -a^{-2}y^{-1}\\
        & 1 & a^{-1}y^{-1}z\\
        & & 1
    \end{pmatrix}\cdot \\ \cdot \begin{pmatrix}
        & & -(1-y^{-1}xz)^{-1}a^{-1}b^{-1}y^{-1}\\
        & 1-y^{-1}xz &\\
        aby & &
    \end{pmatrix} \cdot \\ \cdot \begin{pmatrix}
        1 & b^{-1}y^{-1}x & -b^{-2}y^{-1}\\
        & 1 & (1-y^{-1}xz)^{-1}b^{-1}y^{-1}z\\
        & & 1
    \end{pmatrix} \overline{\psi_s(x)} e_{\psi_s}=\\
    = \sum_{\substack{a, b, y \in \mathbb{F}_q^\times\\x, z \in \mathbb{F}_q\\y \ne xz}}e_{\psi_s} \begin{pmatrix}
        & & -(1-y^{-1}xz)^{-1}a^{-1}b^{-1}y^{-1}\\
        & 1-y^{-1}xz & \\
        aby & &
    \end{pmatrix} \cdot \\ \cdot e_{\psi_s} \overline{\psi_s(x-(1-y^{-1}xz)^{-1}a^{-1}y^{-1}x-b^{-1}y^{-1}x)}.
    \end{align*}\\
    
    To simplify, we can write $d = 1-y^{-1}xz$. We know this is not equal to zero since $y \ne xz$
    \begin{equation*}
        \sum_{\substack{a, b, y \in \mathbb{F}_q^\times\\x, z \in \mathbb{F}_q\\ y \ne xz}}e_{\psi_s} \begin{pmatrix}
        & & -d^{-1}a^{-1}b^{-1}y^{-1}\\
        & d & \\
        aby & &
    \end{pmatrix} e_{\psi_s} \overline{\psi_s(x-d^{-1}a^{-1}y^{-1}x-b^{-1}y^{-1}x)}.
    \end{equation*}

    Now, we do casework on $x$.\\

    \textbf{Subcase 1 ($x = 0$).} If $x = 0$, then $d = 1$ and $x-d^{-1}a^{-1}y^{-1}x-b^{-1}y^{-1}x = 0$. Then, we are left with
    \begin{equation*}
        \sum_{\substack{a, b, y \in \mathbb{F}_q^\times\\z \in \mathbb{F}_q}} e_{\psi_s} \begin{pmatrix}
            & & -a^{-1}b^{-1}y^{-1}\\
            & 1 &\\
            aby & &
        \end{pmatrix}e_{\psi_s} = q(q-1)^2 \sum_{a' \in \mathbb{F}_q^\times} e_{\psi_s} \begin{pmatrix}
            & & -a'^{-1}\\
            & 1 &\\
            a' & &
        \end{pmatrix} e_{\psi_s}.
    \end{equation*}\\
    
    \textbf{Subcase 2 ($x \ne 0$).} If $x \ne 0$, then $d = 1-y^{-1}xz$ now becomes evenly distributed across $\mathbb{F}_q^\times$ (note we exclude zero since $y \ne xz$). Thus, we can replace $z$ with $d$ to get

    \begin{equation*}
        \sum_{a, b, x, y, d \in \mathbb{F}_q^\times} e_{\psi_s} \begin{pmatrix}
            & & -d^{-1}a^{-1}b^{-1}y^{-1}\\
            & d &\\
            aby & &
        \end{pmatrix} e_{\psi_s} \overline{\psi_s(x(1-d^{-1}a^{-1}y^{-1}-b^{-1}y^{-1}))}
    \end{equation*}

    Substituting $\lambda = aby$  to replace $y$ we get

    \begin{equation*}
        \sum_{a, b, x, \lambda, d \in \mathbb{F}_q^\times} e_{\psi_s} \begin{pmatrix}
            & & -d^{-1}\lambda^{-1}\\
            & d &\\
            \lambda & &
        \end{pmatrix} e_{\psi_s} \overline{\psi_s(x(1-d^{-1}\lambda^{-1}b-\lambda^{-1}a))}.
    \end{equation*}

    Summing over $a$, $b$, then $x$ leads to a coefficient of $-1$ and a final result of 

    \begin{equation*}
        -\sum_{\lambda, d \in \mathbb{F}_q^\times} e_{\psi_s} \begin{pmatrix}
            & & -d^{-1}\lambda^{-1}\\
            & d &\\
            \lambda & &
        \end{pmatrix} e_{\psi_s}
    \end{equation*}
    
    When we sum all the values up, we are done with our proof.
\end{proof}

\begin{lem}\label{B_4}
    \begin{multline*}
        q^3\sum_{c, d \in \mathbb{F}_q^\times} e_{\psi_s}
        \begin{pmatrix}
            & & -d^{-1}c^{-1}\\
            & d & \\
            c & &
        \end{pmatrix} e_{\psi_s} \cdot \sum_{a \in \mathbb{F}_q^\times} e_{\psi_s} \begin{pmatrix}
                & & -a^{-1}\\
                & 1 &\\
                a & &
            \end{pmatrix} e_{\psi_s} = \\ = (q-1)\sum_{a \in \mathbb{F}_q^\times} e_{\psi_s} \begin{pmatrix}
            a & & \\
            & a & \\
            & & a^{-2}
        \end{pmatrix}e_{\psi_s} + \\ + (q-1) \sum_{a, b, c \in \mathbb{F}_q^\times}e_{\psi_s} \begin{pmatrix}
            & b &\\
            a & &\\
            & & -a^{-1}b^{-1}
        \end{pmatrix}e_{\psi_s}\overline{\psi_s(bc^{-1}-a^{-1}c)}+\\+(q-1)(q^2-q-1) \sum_{a, b \in \mathbb{F}_q^\times} e_{\psi_s} \begin{pmatrix}
            & & -a^{-1}b^{-1} \\
            & b &\\
            a & &
        \end{pmatrix} e_{\psi_s}.
    \end{multline*}
\end{lem}

\begin{proof}
    We follow a very similar structure to the previous proof.\\

    Expanding, we get
    \begin{align*}
        \sum_{\substack{a, c, d \in \mathbb{F}_q^\times\\x, y, z \in \mathbb{F}_q}}& e_{\psi_s}
        \begin{pmatrix}
            & & -d^{-1}c^{-1}\\
            & d & \\
            c & &
        \end{pmatrix} \begin{pmatrix}
            1 & x & y\\
            & 1 & z\\
            & & 1
        \end{pmatrix} \begin{pmatrix}
                & & -a^{-1}\\
                & 1 &\\
                a & &
        \end{pmatrix} e_{\psi_s} \overline{\psi_s(x)}=\\
        &= e_{\psi_s} \sum_{\substack{a, c, d \in \mathbb{F}_q^\times\\x, y, z \in \mathbb{F}_q}} \begin{pmatrix}
            -ac^{-1}d^{-1} & &\\
            adz & d & \\
            acy & cx & -ca^{-1}
        \end{pmatrix} e_{\psi_s} \overline{\psi_s(x)}.
    \end{align*}

    Again, we split into cases of whether $y = 0$ or not.\\

    \textbf{Case 1 ($y = 0$).} If $y = 0$, $x$ must also equal to zero or else the permutation becomes one that vanishes by the Proposition \ref{GGH_basis}. Thus, we can split into cases based on $z$.\\

    \textbf{Subcase 1.1 ($z = 0$).} This becomes a very simple calculation
    \begin{align*}
        e_{\psi_s} \sum_{a, c, d \in \mathbb{F}_q^\times} \begin{pmatrix}
            -ac^{-1}d^{-1} & &\\
            & d & \\
            & & -ca^{-1}
        \end{pmatrix} e_{\psi_s} &= (q-1) \sum_{a', d \in \mathbb{F}_q^\times}e_{\psi_s} \begin{pmatrix}
            a'^{-1}d^{-1} & &\\
            & d &\\
            & & a'
        \end{pmatrix}e_{\psi_s}=\\
        &= (q-1) \sum_{b \in \mathbb{F}_q^\times} e_{\psi_s} \begin{pmatrix}
            b & & \\
            & b & \\
            & & b^{-2}
        \end{pmatrix} e_{\psi_s}.
    \end{align*}\\
    
    \textbf{Subcase 1.2 ($z \ne 0$).} In this case we have
    \begin{align*}
        &e_{\psi_s} \sum_{a, c, d, z \in \mathbb{F}_q^\times} \begin{pmatrix}
            -ac^{-1}d^{-1} & &\\
            adz & d & \\
            & & -ca^{-1}
        \end{pmatrix} e_{\psi_s}=\\
        &= e_{\psi_s} \sum_{a, c, d, z \in \mathbb{F}_q^\times} \begin{pmatrix}
            1 & -d^{-2}z^{-1}c^{-1} &\\
            & 1 &\\
            & & 1
        \end{pmatrix} \begin{pmatrix}
            & d^{-1}z^{-1}c^{-1} & \\
            adz & & \\
            & & -ca^{-1}
        \end{pmatrix} \begin{pmatrix}
            1 & a^{-1}z^{-1} &\\
            & 1 &\\
            & & 1
        \end{pmatrix} e_{\psi_s}=\\
        &=e_{\psi_s} \sum_{a, c, d, z \in \mathbb{F}_q^\times}
        \begin{pmatrix}
            & d^{-1}z^{-1}c^{-1} & \\
            adz & & \\
            & & -ca^{-1}
        \end{pmatrix} e_{\psi_s} \overline{\psi_s(d^{-2}z^{-1}c^{-1}-a^{-1}z^{-1})}.
    \end{align*}

    Setting $c' = d^{-1}z^{-1}c^{-1}$ and $a' = adz$ we get
    \begin{equation*}
        (q-1)\sum_{a', c', d \in \mathbb{F}_q^\times}e_{\psi_s}
        \begin{pmatrix}
            & c' &\\
            a' & & \\
            & & -a'^{-1}c'^{-1}
        \end{pmatrix}e_{\psi_s} \overline{\psi_s(c'd^{-1}-a'^{-1}d)}.
    \end{equation*}\\
    
    \textbf{Case 2 ($y \ne 0$).} Again, we can assume $y \ne xz$, as then the permutation matrix ensures that this element does not vanish by the Proposition \ref{GGH_basis}. Now, we can write the Bruhat decomposition of our matrices
    \begin{align*}
        e_{\psi_s} \sum_{\substack{a, c, d, y \in \mathbb{F}_q^\times\\x, z \in \mathbb{F}_q\\ y \ne zx}} \begin{pmatrix}
            1 & d^{-2}c^{-1}y^{-1}x(1-y^{-1}zx)^{-1} & -c^{-2}y^{-1}d^{-1}\\
            & 1 & c^{-1}y^{-1}dz\\
            & & 1
        \end{pmatrix} \cdot \\ \cdot \begin{pmatrix}
            & & -d^{-1}a^{-1}c^{-1}y^{-1}(1-y^{-1}zx)^{-1}\\
            & d(1-y^{-1}zx) &\\
            acy & &
        \end{pmatrix} \cdot \\ \cdot \begin{pmatrix}
            1 & a^{-1}y^{-1}x & -a^{-2}y^{-1}\\
            & 1 & (1-y^{-1}zx)^{-1}a^{-1}y^{-1}z\\
            & & 1
        \end{pmatrix} e_{\psi_s} \overline{\psi_s(x)}.
    \end{align*}

    Now, the unipotent part can go into $e_{\psi_s}$, yielding
    \begin{multline*}
        \sum_{\substack{a, c, d, y \in \mathbb{F}_q^\times\\x, z \in \mathbb{F}_q\\ y \ne zx}}e_{\psi_s}\begin{pmatrix}
            & & -d^{-1}a^{-1}c^{-1}y^{-1}(1-y^{-1}zx)^{-1}\\
            & d(1-y^{-1}zx) &\\
            acy & &
        \end{pmatrix}\cdot \\ \cdot e_{\psi_s} \overline{\psi_s(x - d^{-2}c^{-1}y^{-1}x(1-y^{-1}zx)^{-1} - a^{-1}y^{-1}x)}.
    \end{multline*}\\
    
    We do the casework based on $x$.\\

    \textbf{Subcase 2.1 ($x = 0$).} Here, $1-y^{-1}zx = 1$. Then, we have
    \begin{equation*}
        \sum_{\substack{a, c, d, y \in \mathbb{F}_q^\times\\z \in \mathbb{F}_q}}e_{\psi_s}\begin{pmatrix}
            & & -d^{-1}a^{-1}c^{-1}y^{-1}\\
            & d &\\
            acy & &
        \end{pmatrix} e_{\psi_s}.
    \end{equation*}

    Then, we can substitute $a' = acy$ to get
    \begin{equation*}
        q(q-1)^2 \sum_{a', d \in \mathbb{F}_q^\times} e_{\psi_s} \begin{pmatrix}
            & & -d^{-1}a'^{-1}\\
            & d & \\
            a' & &
        \end{pmatrix} e_{\psi_s}.
    \end{equation*}\\
    
    \textbf{Subcase 2.2 ($x \ne 0$).} Then, we can substitute $z' = (1 - y^{-1}zx)$, as $y, x \ne 0$ and $z \in \mathbb{F}_q$, and we can sum over all $z' \in \mathbb{F}_q^\times$
    \begin{equation*}
        \sum_{a, c, d, y, x, z'
        \in \mathbb{F}_q^\times}e_{\psi_s}\begin{pmatrix}
            & & -d^{-1}a^{-1}c^{-1}y^{-1}z'^{-1}\\
            & dz' &\\
            acy & &
        \end{pmatrix} e_{\psi_s} \overline{\psi_s(x (1 - d^{-2}c^{-1}y^{-1}z'^{-1} - a^{-1}y^{-1}))}.
    \end{equation*}

    Substituting $\lambda = acy$ to replace $y$ and $\mu = dz'$ to replace $z'$, we get

    \begin{equation*}
        \sum_{a, c, d, \lambda, x, \mu
        \in \mathbb{F}_q^\times}e_{\psi_s}\begin{pmatrix}
            & & -\lambda^{-1}\mu^{-1}\\
            & \mu &\\
            \lambda & &
        \end{pmatrix} e_{\psi_s} \overline{\psi_s(x (1 - d^{-1}\lambda^{-1}a\mu^{-1} - \lambda^{-1}c))}.
    \end{equation*}

    Summing over $c$, $a$ and $x$ we arrive at a total coefficient of $-1$ and are left with

    \begin{equation*}
        -\sum_{ d, \lambda, \mu
        \in \mathbb{F}_q^\times}e_{\psi_s}\begin{pmatrix}
            & & -\lambda^{-1}\mu^{-1}\\
            & \mu &\\
            \lambda & &
        \end{pmatrix} e_{\psi_s} = -(q-1) \sum_{\lambda, \mu \in \mathbb{F}_q^\times} e_{\psi_s}\begin{pmatrix}
            & & -\lambda^{-1}\mu^{-1}\\
            & \mu &\\
            \lambda & &
        \end{pmatrix} e_{\psi_s}
    \end{equation*}

    When we combine all the equations, we get
    \begin{align*}
        (q-1)\sum_{a \in \mathbb{F}_q^\times} e_{\psi_s} \begin{pmatrix}
            a & & \\
            & a & \\
            & & a^{-2}
        \end{pmatrix}e_{\psi_s} + (q-1) \sum_{a, b, c \in \mathbb{F}_q^\times} e_{\psi_s}\begin{pmatrix}
            & b &\\
            a & &\\
            & & -a^{-1}b^{-1}
        \end{pmatrix}e_{\psi_s}\overline{\psi_s(bc^{-1}-a^{-1}c)}+\\+(q-1)(q^2-q-1) \sum_{a, b \in \mathbb{F}_q^\times} e_{\psi_s} \begin{pmatrix}
            & & -a^{-1}b^{-1} \\
            & b &\\
            a & &
        \end{pmatrix} e_{\psi_s}.
    \end{align*}
\end{proof}

\begin{lem}\label{B_5}
    \begin{align*}
        q^3\left(\sum_{a, b \in \mathbb{F}_q^\times}e_{\psi_s} \begin{pmatrix}
            & a &\\
            -b & &\\
            & & a^{-1}b^{-1}
        \end{pmatrix} e_{\psi_s} \overline{\psi_s(a+b^{-1})} \right)\cdot \left(\sum_{c \in \mathbb{F}_q^\times} e_{\psi_s} \begin{pmatrix}
            & & -c^{-1}\\
            & 1 &\\
            c & &
        \end{pmatrix}e_{\psi_s}\right) =\\= q^2 \sum_{c', b' \in \mathbb{F}_q^\times}e_{\psi_s} \begin{pmatrix}
            & & -c'^{-1}b'^{-1}\\
            & b' & \\
            c' & &\\
        \end{pmatrix} e_{\psi_s}+q^2((q-1)^2-1) \sum_{a \in \mathbb{F}_q^\times}e_{\psi_s} \begin{pmatrix}
            & & -c'^{-1}\\
            & 1 & \\
            c' & &\\
        \end{pmatrix} e_{\psi_s}.
    \end{align*}
\end{lem}

\begin{proof}
    We start by expanding
    \begin{equation*}
        \sum_{\substack{a, b, c \in \mathbb{F}_q^\times\\x, y, z, \in \mathbb{F}_q}}e_{\psi_s} \begin{pmatrix}
            & a &\\
            -b & &\\
            & & a^{-1}b^{-1}
        \end{pmatrix} \begin{pmatrix}
            1 & x & y\\
            & 1 & z\\
            & & 1
        \end{pmatrix} \begin{pmatrix}
            & & -c^{-1}\\
            & 1 &\\
            c & &
        \end{pmatrix} e_{\psi_s} \overline{\psi_s(a+b^{-1}+x)}=
    \end{equation*}

    \begin{equation*}
        =\sum_{\substack{a, b, c \in \mathbb{F}_q^\times\\x, y, z \in \mathbb{F}_q}}e_{\psi_s} \begin{pmatrix}
            acz & a &\\
            -bcy & -bx & bc^{-1}\\
            ca^{-1}b^{-1} & &
        \end{pmatrix} e_{\psi_s} \overline{\psi_s(a+b^{-1}+x)}.
    \end{equation*}

    Note that if $-bx = 0$, this turns into a permutation that vanishes by \ref{GGH_basis}. Thus, we can assume $x \in \mathbb{F}_q^\times$ and that $x$ is invertible. This allows us to find the Bruhat decomposition
    \begin{align*}
        \sum_{\substack{a, b, c, x \in \mathbb{F}_q^\times\\y, z \in \mathbb{F}_q}}e_{\psi_s}& \begin{pmatrix}
            1 & -b^{-1}x^{-1}a & a^2 bz\\
            & 1 & -b^2ay\\
            & & 1
        \end{pmatrix} \begin{pmatrix}
            & & c^{-1}x^{-1}a\\
            & -bx & \\
            a^{-1}b^{-1}c & &\\
        \end{pmatrix} \begin{pmatrix}
            1 & &\\
            & 1 & -x^{-1}c^{-1}\\
            & & 1
        \end{pmatrix} e_{\psi_s}\cdot \\ \cdot \overline{\psi_s(a+b^{-1}+x)}&
        =\sum_{\substack{a, b, c, x \in \mathbb{F}_q^\times\\y, z \in \mathbb{F}_q}}e_{\psi_s} \begin{pmatrix}
            & & c^{-1}x^{-1}a\\
            & -bx & \\
            a^{-1}b^{-1}c & &\\
        \end{pmatrix} e_{\psi_s} \overline{\psi_s(a+b^{-1}+x+b^{-1}x^{-1}a)}.
    \end{align*}
    
    We can substitute $c' = a^{-1}b^{-1}c$ and get rid of the variables $y, z$

    \begin{equation*}
        q^2 \sum_{a, b, c', x \in \mathbb{F}_q^\times}e_{\psi_s} \begin{pmatrix}
            & & c'^{-1}x^{-1}b^{-1}\\
            & -bx & \\
            c' & &\\
        \end{pmatrix} e_{\psi_s} \overline{\psi_s(a(1+b^{-1}x^{-1})+b^{-1}+x)}.
    \end{equation*}

    Now, we do cases.\\

    \textbf{Case 1 ($bx = -1$).} Then, $1+b^{-1}x^{-1} = 0$ and $x = -b^{-1}$
    \begin{equation*}
        q^2 \sum_{a, b, c' \in \mathbb{F}_q^\times}e_{\psi_s} \begin{pmatrix}
            & & -c'^{-1}\\
            & 1 & \\
            c' & &\\
        \end{pmatrix} e_{\psi_s} \overline{\psi_s(-b^{-1}+b^{-1})} = q^2(q-1)^2 \sum_{c'\in \mathbb{F}_q^\times} e_{\psi_s} \begin{pmatrix}
            & & -c'^{-1}\\
            & 1 &\\
            c' & &
        \end{pmatrix}
        e_{\psi_s}.
    \end{equation*}

    \textbf{Case 2 ($bx \ne -1$).} Then, summing over all $a \in \mathbb{F}_q^\times$, $a(1+b^{-1}x^{-1})$ takes on all nonzero values, and we are left with
    \begin{equation*}
        -q^2 \sum_{\substack{b, c', x \in \mathbb{F}_q^\times\\bx \ne -1}}e_{\psi_s} \begin{pmatrix}
            & & c'^{-1}x^{-1}b^{-1}\\
            & -bx & \\
            c' & &\\
        \end{pmatrix} e_{\psi_s} \overline{\psi_s(b^{-1}+x)}.
    \end{equation*}

    There are $(q-2)(q-1)$ different $b, x$ that satisfy $bx \ne -1$, and in addition, $b^{-1}+x$ go over all nonzero values for every value of $-bx$ (can be seen by $b \to k^{-1}b$ and $x \to kx$), so this simplifies to
    \begin{equation*}
        q^2 \sum_{\substack{c', b' \in \mathbb{F}_q^\times\\b' \ne 1}}e_{\psi_s} \begin{pmatrix}
            & & -c'^{-1}b'^{-1}\\
            & b' & \\
            c' & &\\
        \end{pmatrix} e_{\psi_s}.
    \end{equation*}

    We can get rid of the $b' \ne 1$ condition by adding in
    \begin{equation*}
        q^2 \sum_{a \in \mathbb{F}_q^\times}e_{\psi_s} \begin{pmatrix}
            & & -c'^{-1}\\
            & 1 & \\
            c' & &\\
        \end{pmatrix} e_{\psi_s}
    \end{equation*}
     to get a total of
     \begin{equation*}
        q^2 \sum_{c', b' \in \mathbb{F}_q^\times}e_{\psi_s} \begin{pmatrix}
            & & -c'^{-1}b'^{-1}\\
            & b' & \\
            c' & &\\
        \end{pmatrix} e_{\psi_s}-q^2 \sum_{a \in \mathbb{F}_q^\times}e_{\psi_s} \begin{pmatrix}
            & & -c'^{-1}\\
            & 1 & \\
            c' & &\\
        \end{pmatrix} e_{\psi_s}.
    \end{equation*}

    Adding in the first case, we get
    \begin{equation*}
        q^2 \sum_{c', b' \in \mathbb{F}_q^\times}e_{\psi_s} \begin{pmatrix}
            & & -c'^{-1}b'^{-1}\\
            & b' & \\
            c' & &\\
        \end{pmatrix} e_{\psi_s}+q^2((q-1)^2-1) \sum_{a \in \mathbb{F}_q^\times}e_{\psi_s} \begin{pmatrix}
            & & -c'^{-1}\\
            & 1 & \\
            c' & &\\
        \end{pmatrix} e_{\psi_s}.
    \end{equation*}
\end{proof}

\begin{lem}\label{B_6}
    \begin{multline*}
        q^3\left(\sum_{a, b,d \in \mathbb{F}_q^\times}e_{\psi_s} \begin{pmatrix}
            & a &\\
            -b & &\\
            & & a^{-1}b^{-1}
        \end{pmatrix} e_{\psi_s} \overline{\psi_s(ad^{-1}+b^{-1}d)} \right) \left(\sum_{c \in \mathbb{F}_q^\times} e_{\psi_s} \begin{pmatrix}
            & & -c^{-1}\\
            & 1 &\\
            c & &
        \end{pmatrix} e_{\psi_s}\right) =
        \\ = q^2(q^2-q-1)\sum_{a',b'\in\mathbb{F}_q^{\times}}e_{\psi_s} \begin{pmatrix}
            & & -(a'b')^{-1}\\
            & a' & \\
            b' & &\\
        \end{pmatrix} e_{\psi_s}. 
    \end{multline*}
\end{lem}
\begin{proof}
    Using the Bruhat decomposition from the Lemma \ref{B_5}, we get a similar expansion for the product
    \begin{equation*}
        \sum_{\substack{a, b, c,d, x \in \mathbb{F}_q^\times\\y, z, \in \mathbb{F}_q}}e_{\psi_s} \begin{pmatrix}
            & & c^{-1}x^{-1}a\\
            & -bx & \\
            a^{-1}b^{-1}c & &\\
        \end{pmatrix} e_{\psi_s} \overline{\psi_s(ad^{-1}+b^{-1}d+x+b^{-1}x^{-1}a)}.
    \end{equation*}

    Substituting $c' = a^{-1}b^{-1}c$ and $b' = -bx$ we get

    \begin{equation*}
        \sum_{\substack{a, b', c',d, x \in \mathbb{F}_q^\times\\y, z, \in \mathbb{F}_q}}e_{\psi_s} \begin{pmatrix}
            & & -c'^{-1}b'^{-1}\\
            & b' & \\
            c' & &\\
        \end{pmatrix} e_{\psi_s} \overline{\psi_s(ad^{-1}-b'^{-1}xd+x-b'^{-1}a)}=
    \end{equation*}
    \begin{equation*}
        = \sum_{\substack{a, b', c',d, x \in \mathbb{F}_q^\times\\y, z, \in \mathbb{F}_q}}e_{\psi_s} \begin{pmatrix}
            & & -c'^{-1}b'^{-1}\\
            & b' & \\
            c' & &\\
        \end{pmatrix} e_{\psi_s} \overline{\psi_s(a(d^{-1}-b'^{-1})+xd(d^{-1}-b'^{-1}))}.
    \end{equation*}

    If $b' \ne d$, then summing over $a$ and $x$, which are both free variables leads to a total coefficient of $1$, and the sum becomes

    \begin{equation*}
        \sum_{\substack{b', c', d \in \mathbb{F}_q^\times\\y, z, \in \mathbb{F}_q\\b' \ne d}}e_{\psi_s} \begin{pmatrix}
            & & -c'^{-1}b'^{-1}\\
            & b' & \\
            c' & &\\
        \end{pmatrix} e_{\psi_s}= q^2(q-2)\sum_{b', c' \in \mathbb{F}_q^\times}e_{\psi_s} \begin{pmatrix}
            & & -c'^{-1}b'^{-1}\\
            & b' & \\
            c' & &\\
        \end{pmatrix} e_{\psi_s}.
    \end{equation*}

    If $b' = d$, then we are left with

    \begin{equation*}
        \sum_{\substack{a, b', c', x \in \mathbb{F}_q^\times\\y, z, \in \mathbb{F}_q}}e_{\psi_s} \begin{pmatrix}
            & & -c'^{-1}b'^{-1}\\
            & b' & \\
            c' & &\\
        \end{pmatrix} e_{\psi_s}= q^2 (q-1)^2 \sum_{b', c' \in \mathbb{F}_q^\times}e_{\psi_s} \begin{pmatrix}
            & & -c'^{-1}b'^{-1}\\
            & b' & \\
            c' & &\\
        \end{pmatrix} e_{\psi_s}.
    \end{equation*}

    Summing both cases up we get
    
    \begin{equation*}
         q^2(q^2-q-1)\sum_{a',b'\in\mathbb{F}_q^{\times}}e_{\psi_s} \begin{pmatrix}
            & & -(a'b')^{-1}\\
            & a' & \\
            b' & &\\
        \end{pmatrix} e_{\psi_s}.
    \end{equation*}
\end{proof}

\begin{lem}\label{B_2}
    \begin{multline*}
        \left(\sum_{b \in \mathbb{F}_q^\times}e_{\psi_s} \begin{pmatrix}
            b & &\\
            & b &\\
            & & b^{-2}
        \end{pmatrix} e_{\psi_s} \right)\cdot \left(\sum_{a \in \mathbb{F}_q^\times} e_{\psi_s} \begin{pmatrix}
            & & -a^{-1}\\
            & 1 &\\
            a & &
        \end{pmatrix}e_{\psi_s}\right) =\\= \sum_{a', b \in \mathbb{F}_q^\times} e_{\psi_s} \begin{pmatrix}
            & & -a'^{-1}b^{-1}\\
            & b &\\
            a' & &
        \end{pmatrix} e_{\psi_s}.
    \end{multline*}
\end{lem}

\begin{proof}
    We observe that the matrices on the left commute with $e_{\psi_s}$.
\end{proof}

\begin{proof}[Proof of Theorem \ref{eigenvalues_GGH_gl3}.]
    By the Lemmas \ref{B_3}, \ref{B_4}, \ref{B_5}, \ref{B_6} and \ref{B_2}, the transform $\R$ generates a finite-dimensional (with dimension independent on $q$) commutative subalgebra of the Gelfand-Graev Hecke algebra. Let us fix a subspace spanned by six elements $A$ in the Gelfand-Graev Hecke algebra such that the span of $A$ contains the subalgebra generated by $\R$:
    \begin{align*}
        A_1 &:= e_{\psi_s}\begin{pmatrix}
            1 & &\\
            & 1 &\\
            & & 1
        \end{pmatrix}e_{\psi_s},\\
        A_2 &:= \sum_{a \in \mathbb{F}_q^\times} e_{\psi_s} \begin{pmatrix}
            a & & \\
            & a & \\
            & & a^{-2}
        \end{pmatrix}e_{\psi_s},\\
        A_3 &:=\sum_{a \in \mathbb{F}_q^\times} e_{\psi_s} \begin{pmatrix}
            & & -a^{-1}\\
            & 1 &\\
            a & &
        \end{pmatrix} e_{\psi_s},\\
        A_4 &:= \sum_{a, b \in \mathbb{F}_q^\times} e_{\psi_s} \begin{pmatrix}
            & & -a^{-1}b^{-1}\\
            & b &\\
            a & &
        \end{pmatrix} e_{\psi_s},\\
        A_5 &:= \sum_{a, b \in \mathbb{F}_q^\times} e_{\psi_s} \begin{pmatrix}
            & a & \\
            -b & & \\
            & & a^{-1}b^{-1}
        \end{pmatrix} e_{\psi_s} \overline{\psi_s(a+b^{-1})},\\
        A_6 &:= \sum_{a, b,d \in \mathbb{F}_q^\times}e_{\psi_s} \begin{pmatrix}
            & a &\\
            -b & &\\
            & & a^{-1}b^{-1}
        \end{pmatrix} e_{\psi_s} \overline{\psi_s(ad^{-1}+b^{-1}d)}.
    \end{align*}
    Note that due to an explicit description of the basis of Gelfand-Graev Hecke algebras from \cite[Section 4.2]{T}, $A_5$ is linearly independent from $A_6$ because the coefficients in front of the matrices with parameters $a,b$ differ by $p$'th roots of unity, where $q=p^l$ for some $l$, while the respective coefficients in $A_6$ are the real numbers known as the lifted Kloosterman sums which are not zero for $p>2$ due to the distinctness result (see \cite{BB} and references therein). In the case of $p=2,\ l>1$, it is known that the zeroes of Kloosterman sums always exist (\cite{LM09}), i.e. there exists an element $u\in \mathbb{F}_q^{\times}$, such that
    \[
    1+\sum_{d\in \mathbb{F}_q^{\times}}\psi_s(d+u/d)=0.
    \]    
    On the other hand, it is known that $\sum_{d\in \mathbb{F}_q^{\times}}\psi_s(d+u/d) \equiv -1$ (mod $4$) for $p=2$ (\cite{LM09}), so we can deduce the linear independence of $A_5$ and $A_6$ from the divisibility property of the Kloosterman sums whenever $l>3$ (\cite[Theorem 3.6]{LM11}). The values of Kloosterman sums for $\mathbb{F}_4$ and $\mathbb{F}_8$ are known to be $3$ and $\{-5,-1,3\}$ respectively, so we can rule out those cases as well.\\

    According to the Proposition \ref{Radon_transform_decomposition} the summand of the Radon transform corresponding to the Gelfand-Graev Hecke algebra $\mathcal{H}(\GL{3},U_{\lambda_3},\psi_s)$ is
    \[
    q^2e_{\psi_s}+q^3\sum_{a \in \mathbb{F}_q^\times} e_{\psi_s} \begin{pmatrix}
            & & -a^{-1}\\
            & 1 &\\
            a & &
        \end{pmatrix} e_{\psi_s}.
    \]
    
    We can express the action of $\R$ on the elements $A_1$ to $A_6$ by a 6 by 6 matrix
    
    \begin{equation*}
    \begin{blockarray}{cccccc}
        \begin{block}{(cccccc)}
            q^2 & 0                      & q-1       & 0              & 0   & 0\\
            0 & q^2                        & 0     & q-1              & 0   & 0\\
            q^3 & 0                 & q(q-1)^2+q^2       & 0 & q^3(q-2)   & 0\\
            0 & q^3 & -1 & (q-1)(q^2-q-1)+q^2  & q^2 & q^2(q^2-q-1)\\
            0 & 0                      & q-1       & 0              & q^2   & 0\\
            0 & 0                        & 0  & q-1              & 0   & q^2\\
        \end{block}
    \end{blockarray}..
    \end{equation*}

    It yields the following characteristic polynomial
    \[
    (\lambda-1)(\lambda-q)(\lambda-q^2)^2(\lambda-q^3)^2
    \]
    which finishes the proof.
\end{proof}
\begin{rem}
    The list of eigenvalues of the Radon transform remains the same even in characteristic $2$. Indeed, if $q=2^l,\ l>1$, then the only detail we have to change in the Theorem \ref{eigenvalues_gl3_Yok} is the relation $l_i^3=(q-1)^2l_i$ from the Lemma \ref{NCAlgebra_relations}, which simplifies to just $l_i^2=(q-1)l_i$ (then one can check that the minimal polynomial relation still holds true). The proof of the Theorem \ref{eigenvalues_GGH_gl3} remains the same. For the case of $q=2$ we can simply find the eigenvalues numerically.
\end{rem}

\subsection{The case $\lambda=(n-1,1)$ in Ariki-Koike and Yokonuma-Hecke algebras}
In this subsection we deal with the case of $\lambda=(n-1,1)$ within the Yokonuma-Hecke algebra. We observe that the following formula holds in the Yokonuma-Hecke algebra
\[
\R_n = q^{n-1}(1+l_{n-1,n}g_{n-1,n}+ql_{n-2,n}g_{n-2,n}+\dots+ q^{n-2}l_{1,n}g_{1,n}).
\]
Here we denote by $g_{ij}$ the reduced product of generators $g_i$ in the Yokonuma-Hecke algebra corresponding to the transposition $(ij)$. Similarly,
\[
l_{ab}= e_{U_{(1)^n}}\sum_{k=1}^{q-1}t_a^{k}t_{b}^{-k+\frac{q-1}{2}} e_{U_{(1)^n}}.
\]
Without loss of generality we assume $q$ is not a power of $2$ (although the latter case can be treated similarly). Again, we observe that
\[
\R_n = q^n(g_{n-1,n}^2+ ql_{n-2,n}g_{n-2,n}+\dots)=q^2 g_{n-1,n} \R_{n-1} g_{n-1,n},
\]
where $\R_{n-1}$ is viewed as an element of the subalgebra corresponding to $\GL{n-1}$ (in the upper left corner of $\GL{n}$).\\

We note that elements $\R_n,\R_{n-1},l_{n-1,n}$ and $g_{n-1,n}$ satisfy the following relations:
\begin{align*}
    \R_n\R_{n-1}&=\R_{n-1}\R_n,\\
    l_{n-1,n}^3&=(q-1)^2l_{n-1,n},\\
    g_{n-1,n}^2&=q^{-1}(1+l_{n-1,n} g_{n-1,n})
\end{align*}
and $l_{n-1,n}$ commutes with everything in the subalgebra generated by the four elements above.\\

Now, we recall the following (see e.g. \cite{AK}).

\begin{defn}\label{Ariki-Koike_algebra}
    The Ariki-Koike algebra or cyclotomic Hecke algebra $\mathcal{H}_{p,r}$ is an algebra with generators $a_i, i\in \{1,\dots,p\}$ and relations
\begin{align}
    (a_1-u_1)&\cdots (a_1-u_r)=0,\\
    a_i^{2}&=(q-1)a_i+q, \quad i=2,\dots,p,\\
    a_1a_2&a_1a_2=a_2a_1a_2a_1,\\
    &a_ia_j=a_ja_i, \quad |i-j|>1,\\
    a_i &a_{i+1}a_i=a_{i+1}a_ia_{i+1}, \quad i>1.
\end{align}    
\end{defn}

Its representations $V_{\underline{\alpha}}$ are given by multipartitions $\underline{\alpha}=(\alpha^1,\dots,\alpha^r)$ such that $\sum_{i=1}^r |\alpha^i|=n$. The basis of $V_{\alpha}$ is given by the $r$-tuple of Young tableaux $\widehat\alpha$ of the shape $\underline{\alpha}$. For each such a representation the action of $\mathcal{H}_{p,r}(A)$ is given by the following set of rules
\begin{itemize}
    \item The action of $a_1$ is given by
    \begin{equation}
        a_1 \cdot \widehat\alpha = u_{\tau(1)}\widehat\alpha
    \end{equation}
    where $\tau(1)$ is the number of the Young tableaux where $1$ appears.\\

    For $i>1$ we have the following.
    \item If $i-1$ and $i$ appear in the same row of the same tableaux, then
    \begin{equation}
        a_i \cdot \widehat\alpha = q\widehat\alpha.
    \end{equation}
    \item If $i-1$ and $i$ appear in the same column of the same tableaux, then
    \begin{equation}
        a_i \cdot \widehat\alpha = -\widehat\alpha.
    \end{equation}
    \item Otherwise, the action is defined by a $2\times 2$ matrix on the span of $\widehat\alpha$ and $(i-1,i)\cdot \widehat\alpha$ (here $(i-1,i)\cdot \widehat\alpha$ is the Tableaux obtained from $\widehat\alpha$ via a transposition $(i-1,i)$). We have
    \begin{multline}
        a_i\cdot (\widehat{\alpha},(i-1,i)\widehat{\alpha})=(\widehat{\alpha},(i-1,i)\widehat{\alpha})\cdot \frac{1}{1-q^{c(\underline\alpha;i-1)-c(\underline\alpha;i)} \cdot 
         \frac{u_{\tau(i-1)}}{u_{\tau(i)}}}\cdot \\ \cdot \begin{pmatrix}
        q-1 & 1-q^{c(\underline\alpha;i-1)-c(\underline\alpha;i)+1}\frac{u_{\tau(i-1)}}{u_{\tau(i)}}\\
        q-q^{c(\underline\alpha;i-1)-c(\underline\alpha;i)}\frac{u_{\tau(i-1)}}{u_{\tau(i)}} & -q^{c(\underline\alpha;i-1)-c(\underline\alpha;i)} \frac{u_{\tau(i-1)}}{u_{\tau(i)}} (q-1)
    \end{pmatrix}    
    \end{multline}
    Where $c(\underline{\alpha},i)$ is the content of the box $i$.
\end{itemize}
We are interested in the case $p=2$. Namely, consider the algebra homomorphism $f$ from $\mathcal{H}_{2,2n-3}$ over $\CC$ with parameters $u_i=q^{2n-3-i}, 1\le i \le 2n-3$ to the Yokonuma-Hecke algebra:
\[
f: \begin{cases}
    a_1 \mapsto \R_{n-1},\\
    a_2 \mapsto g_{n-1,n},\\
    q \mapsto q\\
\end{cases}.
\]
We will also assume without loss of generality that $l_{n-1,n}=\pm (q-1)$ or $0$.\\

\begin{prop}\label{Ariki_Koike_Yok}
    The eigenvalues of $\R_n$ are contained in a subset $\{1,\dots, q^{2n-2}\}$.
\end{prop}
\begin{proof}
We argue by induction on $n$. From Theorems \ref{eigenvalues_gl2} and \ref{eigenvalues_gl3_Yok} we have the base of the induction. 

Firstly, note that if $l_{n-1,n}=0$, then $g_{n-1,n}^{-1}=qg_{n-1,n}$, so if $\R_{n-1}$ had eigenvalues $1,\dots, q^{2n-4}$, then $\R_n$ will have eigenvalues $q,\dots, q^{2n-3}$.

We have an element $a_1$ and $s:=a_2:=g_{n-1,n}$ satisfying $s^2=q^{-1}(1+ls),l:=l_{n-1,n}$. We wish to compute the action of $\R_n=q^2sa_1s$ on the irreducible representations of $\mathcal{H}_{2,2n-3}$. Assume without loss of generality that $l=(q-1)$, then
\[
s^2 = q^{-1}(1+(q-1)s)
\]
let $\tilde s=qs$
\[
\tilde s^2 = q+ (q-1)\tilde s.
\]
Then $\R_n=\tilde s a_1 \tilde s$. The case of $t=-(q-1)$ is similar: we just have to multiply $\tilde s$ by $-1$ which doesn't change the eigenvalues of $\R_n$. In our case the representations will be either $1$ or $2$ dimensional. 
If we have $\alpha^i=(2)$ for some $i$, then
\[
\tilde sa_1 \tilde s \cdot \widehat\alpha=q^2u_i \widehat\alpha=q^{2+2n-3-i} \widehat\alpha, \quad 1\le i\le 2n-3.
\]
If $\alpha^i=(1,1)$ for some $i$, then
\[
\tilde sa_1 \tilde s \cdot \widehat\alpha=u_i \widehat\alpha=q^{2n-3-i}\widehat\alpha, \quad 1\le i\le 2n-3.
\]
Finally, assume that $\alpha^i=\alpha^j=(1), i\neq j$. Therefore we have to compute
\[
a_2\cdot ((1)^i,(1)^j)=s\cdot ((1)^i,(1)^j), \quad i<j
\]
For $\widehat\alpha$ consisting of two boxes let us fix a basis $x_{ij}$ for $i\neq j$ which corresponds to a Young tableaux with $1$ at the $i$'th box and $2$ at the $j$'th box. We have
\[
\tilde s\cdot (x_{ij},x_{ji})=(x_{ij},x_{ji})\cdot \frac{1}{1-\frac{u_i}{u_j}} \begin{pmatrix}
    q-1 & 1-q\frac{u_i}{u_j}\\
    q-\frac{u_i}{u_j} & -\frac{u_i}{u_j}(q-1)
\end{pmatrix}, \quad i<j.
\]
or
\[
\tilde s\cdot (x_{ij},x_{ji})=\frac{1}{1-\frac{u_i}{u_j}} ((q-1)x_{ij}+(q-\frac{u_i}{u_j})x_{ji},(1-q\frac{u_i}{u_j})x_{ij}-\frac{u_i}{u_j}(q-1)x_{ji}).
\]
then
\[
q^{-1} \tilde sa_1 \tilde s\cdot (x_{ij},x_{ji})=q^{-1}(x_{ij},x_{ji}) \cdot \begin{pmatrix}
    \frac{(q-1)u_j}{u_j-u_i} & \frac{u_j-qu_i}{u_j-u_i}\\
    \frac{qu_j-u_i}{u_j-u_i} & -\frac{u_i(q-1)}{u_j-u_i}
\end{pmatrix}
\cdot \begin{pmatrix}
    u_i & 0\\
    0 & u_j
\end{pmatrix} \cdot \begin{pmatrix}
    \frac{(q-1)u_j}{u_j-u_i} & \frac{u_j-qu_i}{u_j-u_i}\\
    \frac{qu_j-u_i}{u_j-u_i} & -\frac{u_i(q-1)}{u_j-u_i}
\end{pmatrix}
\]
A simple check shows that
\[
q^{-1}\tilde sa_1 \tilde s\cdot (x_{ij},x_{ji})=(u_jx_{ij},u_ix_{ji}).
\]
Note that in our case the representations above are well-defined and irreducible. Since $2^2 \cdot {r\choose 2} + r+r=2r^2 = \dim(\mathcal{H}_{2,r})$ they give the full set of irreducible representations. Thus, the eigenvalues of $\tilde R_n= \tilde sa_1 \tilde s$ are given by $1,\dots,q^{2n-1-i}, 1\le i\le 2n-1$.
\end{proof}
We have the following.
\begin{cor}\label{triang_unipotent_Yok}
    The eigenvalues of the Radon transform for the case $(1)^n$ in the Yokonuma-Hecke algebra form a subset of $\{1,q,\dots, q^{2n-2}\}$.
\end{cor}

\section{The case $\lambda=(k,n-k)$}
According to the Theorem \ref{eigenvalues_box}, the case of the partition $\lambda=(k,n-k)$ plays a central role in the description of eigenvalues of the Radon transform for other partitions. In this Section we provide several results that might be useful in the future research and state the Conjecture \ref{Conjecture}. We assume without loss of generality that $k\le n-k$.\\

\subsection{Basis Description for $\mathcal{H}(\GL{n},U_{(k,n-k)})$} Here we give an explicit basis description for $\mathcal{H}(\GL{n},U_{(k,n-k)})$.\\

Let us consider the action of the group $\GL{n}$ on an $n$-dimensional vector space $\mathbb{F}_q^n$ over $\mathbb{F}_q$ and let us fix a decomposition $\mathbb{F}_q^n = V\oplus W$ where $\dim_{\mathbb{F}_q}(V)=k,\dim_{\mathbb{F}_q}(W)=n-k$. Suppose that $g\in \GL{n}$ is an arbitrary element of the form
\begin{equation}\label{block_matrix}
    g=\begin{pmatrix}
    A & B\\
    C & D
\end{pmatrix}
\end{equation}
where 
\[
    A\in \Hom (V,V),\ C\in \Hom(V,W),\ B\in \Hom(W,V) \text{ and } D\in \Hom(W,W).
\]
The left and right $U_{(k,n-k)}$-action on $g$ allows us to add rows of $C,D$ to $A$ and $B$ correspondingly and to add columns of $A,C$ to $B$ and $D$ respectively.\\

Evidently, the basis of $\mathcal{H}(\GL{n},U_{(k,n-k)})=e_{U_{(k,n-k)}}\CC[\GL{n}]e_{U_{(k,n-k)}}$ is enumerated by the double coset $U_{(k,n-k)}\backslash \GL{n}/ U_{(k,n-k)}$. However, we can give a more precise description of such data in terms of subspaces $V$ and $W$.\\

\begin{prop}\label{Hecke_basis_k_n-k}
    The elements of the double coset $U_{(k,n-k)}\backslash \GL{n}/ U_{(k,n-k)}$ are in $1$ to $1$ correspondence with the following data:
    \[
    C\in \Hom(V,W),\ A\in \Hom(\ker C,V),\ D\in \Hom(W,W/\im C),\ B\in \Hom(\ker D, V/\im A)
    \]
    such that $A$ is injective, $D$ is surjective and $B$ is an isomorphism.
\end{prop}

\begin{proof}
    Firstly, note that the data above is correctly (i.e. uniquely) defined in view of the left and right $U_{(k,n-k)}$-action. If we look at the matrix $g$ from \eqref{block_matrix}, we can note that $U_{(k,n-k)}$ does not change $C$.\\

    Let us fix some complementary subspaces $V'\subset V$ and $W'\subset W$ such that $V=\ker C \oplus V'$ and $W=\im C \oplus W'$. Then there is a unique representative of $A\in \Hom(V,V)$ up to a left $U_{(k,n-k)}$-action such that $A|_{V'}=0$ (since $g\in \GL{n}$, it follows that $A|_{\ker C}$ must be injective). Also, there is a unique representative of $D$ up to the right $U_{(k,n-k)}$-action such that $D$ maps onto $W'$.\\

    Now, the maximal subgroup of $U_{(k,n-k)}$ that fixes $A$ as above under the left action is given by all the elements as follows
    \[
    \begin{pmatrix}
    1 & X\\
    0 & 1
\end{pmatrix}, \quad X|_{\im C}=0.
    \]
    The maximal subgroup of $U_{(k,n-k)}$ fixing $D$ as above under the right action is given by similar elements with a condition that $\im X \subset \ker C$. We are allowed to act by both subgroups on $g$, but in order to see the last piece of data we will need a different decomposition of $V$ and $W$.\\

    Let $W=\ker D \oplus W''$ and $V=\im A \oplus V''$. We consider
    \[
    D: \ker D \oplus W'' \rightarrow \im C \oplus W'.
    \]
    There exists a unique representative $B$ (considering only the right action) under the right action of the first subgroup such that $B|_{W''}=0$. Similarly, there is a unique $B$ under the left action of the second subgroup such that $B$ maps onto $V''$.\\

    Thus we may find a unique linear map $B: \ker D \rightarrow V''$ corresponding to $g$, which must be an isomorphism due to the rank consideration. By identifying $W/\im C$ with $W'$ and $V/\im A$ with $V''$ we see that the correspondence from the Proposition is bijective.
\end{proof}

\begin{rem}\label{Non-canonical_multiplication}
    The labeling of the basis of $\mathcal{H}(\GL{n},U_{(k,n-k)})$ from the Proposition \ref{Hecke_basis_k_n-k} does not seem to help much with description of the multiplication rule in the algebra. Indeed, consider two elements $g,g'\in \GL{n}$. In order to compute the product of $e_{U_{(k,n-k)}}ge_{U_{(k,n-k)}} \cdot e_{U_{(k,n-k)}} g' e_{U_{(k,n-k)}}$, we need to find the data from the Proposition \ref{Hecke_basis_k_n-k} associated with
    \[
    \begin{pmatrix}
    A & B\\
    C & D
    \end{pmatrix}\cdot
    \begin{pmatrix}
        1 & U\\
        0 & 1
    \end{pmatrix}\cdot
    \begin{pmatrix}
    A' & B'\\
    C' & D'
    \end{pmatrix}=
    \begin{pmatrix}
    AA'+AUC'+BC' & AB'+AUD'+BD'\\
    CA'+CUC'+DC' & CB'+CUD'+DD'
    \end{pmatrix}
    \]
    for any $U$. Here $A,B,C,D,A'$ etc. stand for some lifts of the data corresponding to $g,g'$. The element $CA'+CUC'+DC'$, for example, is supposed to represent an element from $\Hom(V,W)$. However,
    \[
    A'\in \Hom(\ker C', V), \quad C\in \Hom(V,W) \quad \Rightarrow  \quad CA'\in \Hom(\ker C',W),
    \]
    so the term $CA'$ can not be lifted canonically.
\end{rem}

\begin{rem}
    Even if we could find the multiplication table for the Hecke algebra \linebreak $\mathcal{H}(\GL{n},U_{(k,n-k)})$, computation of the eigenvalues of the Radon transform still remains to be a difficult problem. In the case when $n=2$ or $3$ we are able to make explicit computations due to the fact that any partition of $n$ involves only $1$ or $2$ -- this allows us to break up the Radon transform into several commuting blocks associated with one-dimensional characters of the corresponding unipotent $U_2$ subgroups of $\GL{n}$. For partitions with entries $k$ bigger than $2$ the representation theory of the corresponding unipotent subgroups $U_k$ becomes significantly complicated. 
\end{rem}

\subsection{Coefficient Calculation}

Now, we seek to calculate the coefficients of the right action of the Radon transform $\R$ on $\CC[\GL{n}/U_{(k,n-k)}]$ for the case $\lambda = (k,n-k)$.\\

\begin{prop}
    Let $g\in \GL{n}$. Recall from Proposition \ref{Radon Expansion} that
    \begin{equation}\label{Radon_coefficients_prop}
        \R\cdot gU_{\lambda} = \sum_{u\in U_{\lambda},u^{op}\in U_{\lambda}^{op}}guu^{op}U_{\lambda}
    \end{equation}
    where
    \[
    u=\begin{pmatrix}
        1 & A\\
        0 & 1
    \end{pmatrix},
    \quad u^{op}=\begin{pmatrix}
        1 & 0\\
        B & 1
    \end{pmatrix}.
    \]
    For fixed $g,u$ and $u^{op}$ the term $guu^{op}U_{\lambda}$ in \eqref{Radon_coefficients_prop} appears exactly $q^{(k-r)(n-k-r)}$ times where $r$ is the rank of $B$.
\end{prop}

\begin{proof} We will calculate the coefficients in the following way. Given a fixed $g$ and a pair $u, u^{op}$, we will find the number of triples $u_1', u_2'\in U_{\lambda}, u^{\prime op}\in U_{\lambda}^{op}, $ where $gu u^{op} = g u_1' u^{\prime op} u_2'$.\\


Let
\[
    u_1'=\begin{pmatrix}
1 & A' \\
0 & 1 
\end{pmatrix},\quad u_2' = \begin{pmatrix}
1 & C' \\
0 & 1 
\end{pmatrix},\quad 
u^{\prime op}=\begin{pmatrix}
1 & 0 \\
B' & 1 
\end{pmatrix}.
\]

Now,
\begin{equation}\label{u_uop}
    u u^{op} = \begin{pmatrix}
1+AB & A \\
B & 1
\end{pmatrix}.
\end{equation}

On the other side,
\[
    u_1' u^{\prime op} u_2' = \begin{pmatrix}
1+A'B' & A'+C'+A'B'C' \\
B' & 1+B'C'
\end{pmatrix}.
\]

We wish to solve the following equation for $A',B'$ and $C'$ with fixed $A,B$:
\[
    \begin{pmatrix}
1+AB & A \\
B & 1 
\end{pmatrix} = \begin{pmatrix}
1+A'B' & A'+C'+A'B'C' \\
B' & 1+B'C'
\end{pmatrix}.
\]

Notably, we have these equations:
\begin{align*}
    B &= B',\\
    1+AB &= 1+A'B',\\
    1 &= 1 + B'C',\\
    A &= A'+C' + A'B'C'.
\end{align*}

Performing some cancellations and substitutions we obtain
\begin{align*}
    B &= B',\\
    AB &= A'B,\\
    0 &= BC',\\
    A &= C' + A'.
\end{align*}

Note that we can substitute $A = C'+A'$ into the second equation to get 
\[
    (C'+A')B = A'B \ \Leftrightarrow \
    C'B + A'B = A'B \ \Leftrightarrow \
    C'B = 0.
\]

So, we are left with the following
\begin{align*}
    B &= B',\\
    0 &= BC' = C'B,\\
    A &= C' + A'.
\end{align*}

Now, let $rk(B) = r$. In addition, let $B, B'$ be $n-k$ by $k$ matrices and $A, A', C$ be $k$ by $n-k$ matrices. Recall that $A$ and $B$ are fixed matrices and we want to find the number of $A', B'$ and $C'$ that satisfy the above relations.\\

Clearly, we only have one choice for $B'$, as it must be equal to $B$. Also, the choice of $A'$ is based on the choice of $C'$ from the third equation. So, the coefficient only depends on the number of choices of $C'$.\\

Recall the setting of the Proposition \ref{Hecke_basis_k_n-k}. Note that since $C'B = 0$, $C'$ must send the image of $B$ to $0$, so it suffices to determine its action on the complementary space of the image. Specifically, if we take $B$ to be a linear transformation from a vector space $V$ to $W$, we need to determine the action of $C'$ on $W / \im B$. Note that the latter space has dimension $n-k-r$. Since $BC'$ is also $0$, $C'$ must send $W$ to the kernel of $B$ as well. The dimension of the kernel of $B$ is $k-r$, so the number of elements in it is $q^{k-r}$. Therefore, we have $q^{k-r}$ many choices for every basis element in $W / \im B$, so the coefficient is $q^{(k-r)(n-k-r)}$.\\
\end{proof}

It is possible to compute the trace of the Radon transform and the trace of its square.\\

\begin{prop}\label{trace_and_square}
    For the Radon transform acting on $\CC[\GL{n}/U_{(k,n-k)}]$ we have
    \[
    \tr (\R)=|\GL{n}|=\prod_{i=0}^{n-1}(q^n-q^i)
    \]
    and
    \[
    \tr(\R^2)=\prod_{i=0}^{n-1}(q^n-q^i) \sum_{r = 0}^{k}
\frac{\displaystyle\prod_{i=0}^{r-1} (q^k - q^i)\,(q^{n-k} - q^i)}
{\displaystyle\prod_{i=0}^{r-1} (q^r - q^i)} q^{(n-k-r)(k-r)}.
    \]
\end{prop}

\begin{proof}
The first part easily follows from observation of elements on the diagonal. From the previous Proposition we have that all the diagonal coefficients of $\R$ are equal to $q^{k(n-k)}=|U_{(k,n-k)}|$. Therefore, the trace is $|\GL{n} / U_{(k,n-k)}| \cdot q^{k(n-k)} = |\GL{n}| = \prod_{i=0}^{n-1}(q^n-q^i)$.\\


Let $\R$ act by the matrix $a_{ij}$ with respect to the standard basis. Recall from the Proposition \ref{symmetric_Rd} that $a_{ij}=a_{ji}$. The trace of the square of the matrix is $\sum_{i, j} a_{ij}a_{ji}$, but we can rewrite it as $\sum_{i, j} a_{ij}^2$, which is the sum of all the elements squared.\\

The number of times the coefficient $q^{(k-r)(n-k-r)}$ shows up is the number of choices of $A$ that result in distinct outcomes times the number of choices for $B$ of rank $r$. It follows from the last paragraph in the proof of the Proposition \ref{trace_and_square} that the number of such choices for $A$ is the size of the group $U_{(k,n-k)}$ divided by the number of choices for $C'$ (with conditions $BC'=C'B=0$), which is $q^{rk+r(n-k)-r^2}=q^{r(n-r)}$.\\

Then, the number of choices for $B$ is the number of matrices of rank $r$, which is equal to
\[
    \frac{\displaystyle\prod_{i=0}^{r-1} (q^k - q^i)\,(q^{n-k} - q^i)}
{\displaystyle\prod_{i=0}^{r-1} (q^r - q^i)}.
\]
We can see it as follows: in order to specify a linear map $B: \mathbb{F}_q^{k} \rightarrow \mathbb{F}_q^{n-k}$ of rank $r$ firstly we need to find a subspace $L\subset \mathbb{F}^{n-k}$ of dimension $r$, which can be done in 
\[
\begin{bmatrix}
    n-k\\
    r
\end{bmatrix}_q:=\frac{(q^{n-k}-1)\cdots (q^{n-k}-q^{r-1})}{(q^r-1)\cdots (q^r-q^{r-1})}
\]
ways. Then we just have to specify a surjective map $\Bar{B}:\mathbb{F}_q^k \twoheadrightarrow L$. There are
\[
(q^k-1)\cdots (q^k-q^{r-1})
\]
ways to do so.\\

Thus, for every $r$ from $0$ to $\min(k, n-k)=k$, the coefficient $q^{(k-r)(n-k-r)}$ occurs
\[\frac{\displaystyle\prod_{i=0}^{r-1} (q^k - q^i)\,(q^{n-k} - q^i)}
{\displaystyle\prod_{i=0}^{r-1} (q^r - q^i)} q^{r(n-r)}\]
times.\\

Then $\sum_{ij}a_{ij}^2$ should be equal to
\begin{equation*}
|\GL{n} / U_{(k,n-k)}| \sum_{r = 0}^{k}
\frac{\displaystyle\prod_{i=0}^{r-1} (q^k - q^i)\,(q^{n-k} - q^i)}
{\displaystyle\prod_{i=0}^{r-1} (q^r - q^i)} q^{r(n-r)} \cdot \left(q^{(k-r)(n-k-r)}\right)^2=
\end{equation*}
\begin{equation*}
=\prod_{i=0}^{n-1}(q^n-q^i) \sum_{r = 0}^{k}
\frac{\displaystyle\prod_{i=0}^{r-1} (q^k - q^i)\,(q^{n-k} - q^i)}
{\displaystyle\prod_{i=0}^{r-1} (q^r - q^i)} q^{(k-r)(n-k-r)}.
\end{equation*}
\end{proof}

According to what we have observed so far we formulate the following Conjecture.

\begin{conj}\label{Conjecture}
    The eigenvalues of the Radon transform $\R$ for a fixed $\lambda$ acting on $\CC[\GL{n}/U_{\lambda}]$ are given by powers of $q$.
\end{conj}

\begin{rem}
    The Radon transform for the $\lambda=(k,n-k)$ case can be restricted to the Yokonuma-Hecke subalgebra $\mathcal{H}(\GL{n},U_{\lambda_n})$ in the Hecke algebra $\mathcal{H}(\GL{n},U_{(k,n-k)})$. For the purpose of this Remark, let $\R_{\lambda_m},m\ge 1$ be the Radon transform in $\mathcal{H}(\GL{m},U_{\lambda_m})$, $R_{(k,n-k)}$ is the Radon transform in $\mathcal{H}(\GL{n},U_{(k,n-k)})$. If we could prove the Conjecture for the case of $\mathcal{H}(\GL{n},U_{\lambda_n})$, then since
    \begin{equation}\label{block_triang}
        \R_{\lambda_n} = \R_{(k,n-k)} \cdot (\R_{\lambda_k} \cdot \R_{\lambda_{n-k}}),
    \end{equation}
    we can prove the partial eigenvalue Conjecture (i.e. provided that $\R_{\lambda_k}$ and $\R_{\lambda_{n-k}}$ are invertible) for the Radon transform of type $(k,n-k)$ acting on $\mathcal{H}(\GL{n},U_{\lambda_n})$. Here we consider embedding of the algebras $\mathcal{H}(\GL{k},U_{\lambda_k})$ and $\mathcal{H}(\GL{n-k},U_{\lambda_{n-k}})$ into $\mathcal{H}(\GL{n},U_{\lambda_n})$ which correspond to the block-diagonal embedding $\GL{k}\times \GL{n-k}\hookrightarrow \GL{n}$.\\

    Note that due to the Proposition \ref{Ariki_Koike_Yok} and Corollary \ref{triang_unipotent_Yok} we know that the eigenvalues of the Radon transform in the Yokonuma-Hecke algebra of type $\lambda=(1)^n$ are given by integer powers of $q$ (here we also observe that all the eigenvalues in question are non-zero).
\end{rem}



\subsection{Coefficients for the Parabolic Case $\lambda=(k,n-k)$}
We conclude this section with computation of the matrix coefficients for the action of the Radon transform $\R$ on the Hecke algebra $\mathcal{H}(\GL{n}, P_{(k,n-k)})$. Note that the problem is more approachable in this case since $\mathcal{H}(\GL{n},P_{(k,n-k)})$ is a $k+1$-dimensional commutative algebra due to \cite[Theorem 3.1]{CIK}. As usual, we assume that $k\le n-k$.\\

Recall the Bruhat decomposition \eqref{Bruhat_decomposition_B} from the Theorem \ref{bruhat}. Since $B_n\subset P_{(k,n-k)}$, any element $g\in \GL{n}\subset \mathcal{H}(\GL{n},P_{(k,n-k)})$ can be brought to a permutation matrix from $S_n$. Furthermore, left and right action by $\GL{k}\times \GL{n-k}$ on $g$ allows us to bring any permutation matrix to the following form
    \[
    p_i:=\begin{blockarray}{ccccc}
      & k-i & i & i & n-k-i\\
        \begin{block}{c(cccc)}
           k-i & 1 & 0 & 0 & 0\\
           i & 0 & 0 & 1 & 0\\
           i & 0 & 1 & 0 & 0\\
           n-k-i & 0 & 0 & 0 & 1\\
        \end{block}
    \end{blockarray} .
    \]
    for some $0\le i \le k$. Note that left and right multiplication of $g$ on the left or on the right by an element from $P_{(k,n-k)}$ does not change the rank of its $n-k \times k$ submatrix from the lower left corner. Therefore, the elements
    \[
        e_i := e_{P_{(k,n-k)}} p_i e_{P_{(k,n-k)}} \in \mathcal{H}(\GL{n}, P_{(k,n-k)})
    \]
    form a basis of the algebra.\\

\begin{prop}
    The matrix coefficients of $\R$ in $\mathcal{H}(\GL{n}, P_{(k,n-k)})$ with respect to the basis $e_i$ are given by
    \[
    R_{ij}=\sum_{\substack{0\le t\le i,\\ t+s=j, s\ge 0}}q^{(n-k)k-i^2+t(n-t)}\prod_{m=0}^{i-t-1}(\frac{(q^i-q^m)^2}{(q^{i-t}-q^m)})\prod_{l=0}^{s-1}(\frac{(q^{k-t}-q^l)(q^{n-k-t}-q^l)}{(q^{s}-q^l)}).
    \]
\end{prop}
\begin{proof}
    We note that
    \begin{equation}\label{Rd_parabolic}
    \R \cdot e_i = \sum_{\substack{U\in {\rm Mat}_{k \times n-k}(\mathbb{F}_q) \\ U^{op}\in {\rm Mat}_{n-k\times k}(\mathbb{F}_q)}}e_{P_{(k,n-k)}}p_i \cdot \begin{pmatrix}
        1 & U\\
        0 & 1
    \end{pmatrix}\cdot
    \begin{pmatrix}
        1 & 0\\
        U^{op} & 1
    \end{pmatrix}
    e_{P_{(k,n-k)}}.
    \end{equation}
    Let 
    \[
    U= \begin{blockarray}{ccc}
     & i & n-k-i\\
     \begin{block}{c(cc)}
         k-i & A_1 & A_2\\
         i & A_3 & A_4\\
     \end{block}
    \end{blockarray}\ ,
    \quad
    U^{op}=\begin{blockarray}{ccc}
     & k-i & i\\
     \begin{block}{c(cc)}
         i & B_1 & B_2\\
         n-k-i & B_3 & B_4\\
     \end{block}
    \end{blockarray}\ .
    \]

    The lower left $n-k \times k$ submatrix of the matrix between the projectors $e_{P_{(k,n-k)}}$ from \eqref{Rd_parabolic} is given by
    \[
    \begin{blockarray}{ccc}
     & k-i & i\\
     \begin{block}{c(cc)}
         i & A_3B_1+A_4B_3 & 1+A_3B_2+A_4B_4\\
         n-k-i & B_3 & B_4\\
     \end{block}
    \end{blockarray}\ .
    \]
    We may apply
    \[
    \begin{pmatrix}
        1 & -A_4\\
        0 & 1
    \end{pmatrix}
    \]
    to it on the left to get rid of $A_4$
    \[
    \begin{blockarray}{ccc}
     & k-i & i\\
     \begin{block}{c(cc)}
         i & A_3B_1 & 1+A_3B_2\\
         n-k-i & B_3 & B_4\\
     \end{block}
    \end{blockarray}\ ,
    \]
    which does not change the rank of the matrix. Furthermore, by $\GL{i}$ conjugation of $1+A_3B_2$ and introduction of an auxiliary element $h$ for $1+ gA_3 h h^{-1}B_2 g^{-1}$ we may assume without loss of generality that $A_3$ has the form
    \[
    \begin{blockarray}{ccc}
     & r & i-r\\
     \begin{block}{c(cc)}
         r & 1 & 0\\
         i-r & 0 & 0\\
     \end{block}
    \end{blockarray}\ ,
    \]
    where $r={\rm rk} A_3$.\\

    Now, by making an appropriate shift of the matrix $B_2$ and assuming that
    \[
    B_1= \begin{blockarray}{ccc}
     &  & k-i\\
     \begin{block}{cc(c)}
          & r & B_{11}\\
          & i-r & B_{12}\\
     \end{block}
    \end{blockarray}\ ,
    \quad
    B_2=\begin{blockarray}{ccc}
     & r & i-r\\
     \begin{block}{c(cc)}
         r & B_{21} & B_{22}\\
         i-r & B_{23} & B_{24}\\
     \end{block}
    \end{blockarray}\ ,
    B_4=\begin{blockarray}{ccc}
     & r & i-r\\
     \begin{block}{c(cc)}
         n-k-i & B_{41} & B_{42}\\
     \end{block}
    \end{blockarray}\ ,
    \]
    we can bring the former $(n-k)\times k$ matrix to the form
    \[
    \begin{pmatrix}
        B_{11} & B_{21} & B_{22}\\
        0 & 0 & 1\\
        B_3 & B_{41} & B_{42}
    \end{pmatrix}.
    \]
    Therefore, the summand from \eqref{Rd_parabolic} corresponding to such a collection of matrices is equal to $e_j$ where
    \[
    j=i-r +{\rm rk}\begin{pmatrix}
        B_{11} & B_{21}\\
        B_3 & B_{41}
    \end{pmatrix}.
    \]
    
    By observing that $A_1,A_2,A_4,B_{12},B_{23},B_{24}, B_{22}$ and $B_{42}$ are free variables in this formula, we may apply the counting method from the previous subsection to obtain the matrix coefficient of the Radon transform:
    \[
        \R \cdot e_i =\sum_{j=0}^k\sum_{\substack{0\le t\le i,\\ t+s=j}}q^{(n-k)k-i^2+t(n-t)}\prod_{m=0}^{i-t-1}(\frac{(q^i-q^m)^2}{(q^{i-t}-q^m)})\prod_{l=0}^{s-1}(\frac{(q^{k-t}-q^l)(q^{n-k-t}-q^l)}{(q^{s}-q^l)})e_j.
    \]
\end{proof}

Now we prove the following.

\begin{prop}\label{max_parabolic}
    The eigenvalues of the matrix $R_{ij}$ are given by
    \[
    q^{k(n-k)-k+(n-2k+1)t+t^2}, \qquad 0\le t \le k.
    \]
\end{prop}
\begin{proof}
    Let us consider the natural right action of $\mathcal{H}(\GL{n},P_{(k,n-k)})$ on $\CC[\GL{n}/P_{(k,n-k)}]$. Recall that the quotient $\GL{n}/P_{(k,n-k)}$ can be identified with the Grassmann scheme $Gr_q(n, k)$ of all $k$-dimensional $\mathbb{F}_q$-subspaces in $\mathbb{F}_q^n$. In view of the Remark \ref{convol_Hecke_alg} and Proposition \ref{convol_Hecke_mod} we may identify $\CC[\GL{n}/P_{(k,n-k)}]$ with the space of functions on $Gr_q(n, k)$, then the generator $e_i$ acts on it proportionally to the following operator
    \begin{equation}\label{Bose-Mesner_gens}
        (A_i \cdot f)(L)=\sum_{\dim(M\cap L)=k-i}f(M), \quad L\in Gr_q(n,k).
    \end{equation}
    By comparing both actions on a constant function one can see that
    \[
        A_i=v_i e_i
    \]
    where
    \[
    v_i :=\#\{M\in Gr_q(n,k)| \dim(M\cap L)=k-i\}=q^{i^2}\begin{bmatrix}
    k\\
    i
\end{bmatrix}_q \begin{bmatrix}
    n-k\\
    i
\end{bmatrix}_q
    \]
    for a fixed $L\in Gr_q(n,k)$. The operators $A_i$ from \eqref{Bose-Mesner_gens} form the well known $q$-Johnson Bose-Mesner algebra and their eigenvalues on the space of functions above are given by $q$-Eberlein polynomials (see e.g. \cite{LM} and \cite{W})
    \begin{equation}\label{q_Eberlein}
        A_i E_j=P_i(j)E_j, \quad 0\le i,j \le k
    \end{equation}
    where $E_j$ are the primitive idempotents of the algebra and 
    \[
    P_i(j)=\sum_{s=0}^i (-1)^{i-s}q^{sj+{i-s\choose 2}}\begin{bmatrix}
    k-s\\
    i-s
\end{bmatrix}_q \begin{bmatrix}
    k-j\\
    s
\end{bmatrix}_q \begin{bmatrix}
    n-k-j+s\\
    s
\end{bmatrix}_q.
    \]

    Now we can find the expansion of $\R$ in terms of the generators $A_i$. Just like in Proposition \ref{Hecke_basis_k_n-k} we can fix a decomposition $\mathbb{F}_q^n=V\oplus W$. The image of $V$ under an element $uu^{op}$ in \eqref{u_uop} is the $n\times k$ column space
    \[
    \begin{pmatrix}
        1+AB\\
        B
    \end{pmatrix}
    \]
    Let us fix without loss of generality a $k$-dimensional subspace $V'\subset\mathbb{F}_q^n$, such that $\dim(V\cap V')=k-i$ and 
    \[
    V=K\oplus L_1, \quad V' = K \oplus Z_1, \quad Z_1 \subset W.
    \]
    Here $V\cap V'=K$. Since we have $uu^{op}(V)=V'$, the map $B:V\rightarrow W$ should be $0$ on $K$ and induce an isomorphism between $L_1$ and $Z_1$. Therefore, there are exactly
    \[
    |\GL{i}|=\prod_{m=0}^{i-1}(q^i-q^m)
    \]
    choices of $B$. After fixing $B$, we note that the restriction of $AB$ to $L_1$ must be $-1$ modulo $K$ because $Z_1\subset W$, so $A|_{Z_1}=-B^{-1}|_{Z_1}$. Then there are $q^{i(k-i)}$ choices for $A|_{Z_1}$ and, independently, $q^{k(n-k-i)}$ choices for the values of $A$ on a complement of $Z_1$ in $W$. Hence, the total number of pairs $(A,B)$ sending $V$ to a given space $V'$ such that $\dim(V\cap V')=k-i$ is equal to
    \[
    q^{k(n-k)-i^2}|\GL{i}|
    \]
    and
    \begin{equation}\label{BM_expression_for_Radon}
        \R = \sum_{i=0}^k q^{k(n-k)-i^2}|\GL{i}| A_i.
    \end{equation}

    By using \eqref{q_Eberlein} we can compute the eigenvalues by which $\R$ acts on the space of functions on $Gr_q(n,k)$. Namely, $\R E_j = \lambda_j E_j$
    \begin{equation}
        \lambda_j = \sum_{i=0}^{k}q^{k(n-k)-i^2}|\GL{i}|\sum_{s=0}^i (-1)^{i-s}q^{sj+{i-s\choose 2}}\begin{bmatrix}
    k-s\\
    i-s
\end{bmatrix}_q \begin{bmatrix}
    k-j\\
    s
\end{bmatrix}_q \begin{bmatrix}
    n-k-j+s\\
    s
\end{bmatrix}_q.
    \end{equation}
    We can re-write it as
    \[
    \lambda_j = q^{k(n-k)}\sum_{s'=0}^{k-j}q^{s'j}\begin{bmatrix}
    k-j\\
    s'
\end{bmatrix}_q \begin{bmatrix}
    n-k-j+s'\\
    s'
\end{bmatrix}_q
\sum_{i'=0}^{k-s'}(-1)^{i'}q^{{i' \choose 2}-(s'+i')^2}|\GL{s'+i'}|\begin{bmatrix}
    k-s'\\
    i'
\end{bmatrix}_q
    \]
    Note that
    \[
    q^{-(s'+i')^2}|\GL{s'+i'}|=q^{-(s')^2}|\GL{s'}|(1-q^{-s'-1})\cdots(1-q^{-s'-i'}).
    \]
    Let us denote
    \begin{equation}\label{q-Pochhammer}
        (a;q)_n:=\prod_{i=0}^{n-1}(1-aq^i),
    \end{equation}
    then the inner sum becomes
    \[
    q^{-(s')^2}|\GL{s'}|\sum_{i'=0}^{k-s'}(-1)^{i'}q^{{i' \choose 2}}(1-q^{-s'-1})\cdots(1-q^{-s'-i'})\begin{bmatrix}
    k-s'\\
    i'
\end{bmatrix}_q=
    \]
    \[
    =q^{-(s')^2}|\GL{s'}|\sum_{i'=0}^{k-s'}q^{{i' \choose 2}}q^{-\frac{i'(i'+1)}{2}} \frac{(q^{-s'-1};q^{-1})_{i'} (q^{-(s'-k)};q^{-1})_{i'}}{(q^{-1};q^{-1})_{i'}}=
    \]
    \[
    =q^{-(s')^2}|\GL{s'}|\sum_{i'=0}^{k-s'} \frac{(q^{-s'-1};q^{-1})_{i'} (q^{-(s'-k)};q^{-1})_{i'}}{(q^{-1};q^{-1})_{i'}}q^{-i'}=q^{-(s')^2}|\GL{s'}| q^{-(s'+1)(k-s')}
    \]
    due to the $q$-Chu–Vandermonde identity (see \cite[Appendix II.6]{GM}).\\

    Now we have
    \[
     \lambda_j = q^{k(n-k)}\sum_{s'=0}^{k-j}q^{s'j}\begin{bmatrix}
    k-j\\
    s'
\end{bmatrix}_q \begin{bmatrix}
    n-k-j+s'\\
    s'
\end{bmatrix}_q
q^{-(s')^2}|\GL{s'}|q^{-(s'+1)(k-s')}=
    \]
    \[
    =q^{k(n-k-1)}\sum_{s'=0}^{k-j}q^{s'(1-(k-j))}\begin{bmatrix}
    k-j\\
    s'
\end{bmatrix}_q \begin{bmatrix}
    n-k-j+s'\\
    s'
\end{bmatrix}_q
|\GL{s'}|=
    \]
    \[
    =q^{k(n-k-1)}\sum_{s'=0}^{k-j}q^{s'}(q^{-(k-j)};q)_{s'} \begin{bmatrix}
    n-k-j+s'\\
    s'
\end{bmatrix}_q=q^{k(n-k-1)}\sum_{s'=0}^{k-j}\frac{(q^{-(k-j)};q)_{s'}(q^{n-k-j+1};q)_{s'}}{(q;q)_{s'}}q^{s'}=
    \]
    \[
    =q^{k(n-k-1)}q^{(n-k-j+1)(k-j)}.
    \]
    Observe that if $t=k-j$, then
    \[
    \lambda_{k-t}=q^{k(n-k-1)+(n-2k+1+t)t}=q^{k(n-k)-k+(n-2k+1)t+t^2}, \quad 0\le t \le k,
    \]
    and we are done.
\end{proof}

\appendix

\section{Verification of the Polynomial Relation in the Yokonuma-Hecke Algebra}
In this section we demonstrate the Wolfram computation verifying the polynomial relation \eqref{Yok_Rd_relation}. In what follows, we fix some notations
\begin{equation}
    x=g_1,\ y=g_2, Q=q, t=\sum_{i=1}^{q-1}t_1^i t_2^{(q-1)/2-i}, s=\sum_{i=1}^{q-1}t_1^i t_3^{(q-1)/2-i},r=\sum_{i=1}^{q-1}t_2^i t_3^{(q-1)/2-i}.
\end{equation}
Here is the clean Wolfram code that yields $0$ as the result of the verification.




\begin{lstlisting}[language=Mathematica, basicstyle=\tiny]
PacletInstall["https://github.com/NCAlgebra/NC/blob/master/NCAlgebra-6.0.3.paclet?raw=\true",
ForceVersionInstall -> True];
<< NCAlgebra`
<< NCGBX`
Clear[s]
SetNonCommutative[s]
SetMonomialOrder[{t, s, r}, {x, y}];
PrintMonomialOrder[];
rels = {
  x ** y ** x - y ** x ** y,
  
  Q ** x ** x - 1 - t ** x, 
  Q ** y ** y - 1 - r ** y,
  
  x ** t - t ** x,
  y ** t - s ** y,
  x ** s - r ** x,
  y ** s - t ** y,
  x ** r - s ** x,
  y ** r - r ** y,
  
  s ** s ** s - (Q - 1) ** (Q - 1) ** s,
  r ** r ** r - (Q - 1) ** (Q - 1) ** r,
  t ** t ** t - (Q - 1) ** (Q - 1) ** t,
  
  s ** t - t ** s,
  t ** r - r ** t,
  s ** r - r ** s,
  
  s ** t - s ** r,
  r ** s - r ** t,
  t ** s - t ** r}
gb = NCMakeGB[rels, 1000]

NCExpandReplaceRepeated[(Q^2 + Q^2 ** t ** x + 
    Q^3 ** s ** x ** y ** x - 1) ** (Q^2 + Q^2 ** t ** x + 
    Q^3 ** s ** x ** y ** x - Q) ** (Q^2 + Q^2 ** t ** x + 
    Q^3 ** s ** x ** y ** x - Q^2) ** (Q^2 + Q^2 ** t ** x + 
    Q^3 ** s ** x ** y ** x - Q^3) ** (Q^2 + Q^2 ** t ** x + 
    Q^3 ** s ** x ** y ** x - Q^4), gb]
\end{lstlisting}

\end{document}